\documentclass[11pt,a4paper]{amsart}
\usepackage{geometry}
\usepackage[colorlinks, linkcolor=red, anchorcolor=blue, citecolor=blue]{hyperref}
\usepackage{color}             
\usepackage{amsthm,bm} %bm make \mathbf work%
\usepackage{enumerate}
\usepackage{amsfonts}
\usepackage{amsmath}
\usepackage{mathrsfs}
\usepackage{amssymb}
\usepackage{amsbsy}
\usepackage{ifthen}
\usepackage[all]{xy}
\usepackage{extarrows}
\usepackage{indentfirst}        %éŠè¡çŒ©è¿å®å
\usepackage{latexsym}        % å€çæ°å­Šå

\usepackage{graphicx}
\usepackage{cases}
\usepackage{pifont}
\usepackage{txfonts}
\usepackage{xcolor}
\usepackage{multirow}
\usepackage{caption, subcaption}
\usepackage{tikz}

\newcounter{maint}

\usepackage[vcentermath]{youngtab}
\numberwithin{equation}{section}

\newcommand{\ydh}{{}^H_H\mathcal{YD}}

\allowdisplaybreaks[4] %% allow to display a terrible long formulae which can't be put in one paper. 

\begin{document}

\newtheorem{theorem}{Theorem}[section]

\newtheorem{lemma}[theorem]{Lemma}

\newtheorem{corollary}[theorem]{Corollary}
\newtheorem{proposition}[theorem]{Proposition}

\theoremstyle{remark}
\newtheorem{remark}[theorem]{Remark}

\theoremstyle{definition}
\newtheorem{definition}[theorem]{Definition}

\theoremstyle{definition}
\newtheorem{conjecture}[theorem]{Conjecture}

\newtheorem{example}[theorem]{Example}
\newtheorem{problem}[theorem]{Problem}
\newtheorem{question}[theorem]{Question}

%%%%%%%%%%%%%%%%%%%%%%%%%%%%%%%%%%%%%%%%%%%%%%%
%%%%%%%%%%%%%%%%%%%%%%%%%%%%%%%%%%%%%%%%%%%%%%%%%

\def\k{\Bbbk}
\def\id{\mathrm{id}}
\def\ad{\mathrm{ad}}
%\title[Nichols algebras of dimension $n^m$ and Multinomial expansion]
%{Nichols algebras of dimension $n^m$ and Multinomial expansion}
\title[Multinomial expansion and Nichols algebras $\mathfrak{B}(W_{X,r})$]
{Multinomial expansion and Nichols algebras associated to non-degenerate involutive  solutions of the Yang-Baxter equation}

\author[Shi]{Yuxing Shi }
\address{School of Mathematics and Statistics, Jiangxi Normal University,  Nanchang 330022, P. R. China}\email{yxshi@jxnu.edu.cn}

%\subjclass{Primary 17B37, 81R50; Secondary 17B35}
\subjclass[2010]{16T05, 16T25, 17B37}
%\date{October 22, 2013}
\thanks{
%2010 Mathematics Subject Classification: 16T05.\\
\textit{Keywords:} Nichols algebra; Multinomial expansion; Symmetric group; Gelfand-Kirillov dimension;
Hopf algebra.
%\\
%This work was partially supported by
%Foundation of Jiangxi Educational Committee (No.12020447)
}

\begin{abstract}
In this paper, we investigate the Nichols algebra $\mathfrak{B}(W_{X,r})$
associated to any non-degenerate involutive solution $(X, r)$ of the Yang-Baxter equation.  
Infinite examples of finite dimensional Nichols algebras are obtained, including 
those of dimension $n^m$ with $m$, $n\in\Bbb Z^{\geq 2}$. 
It turns out that the Nichols algebra $\mathfrak{B}(W_{X, r})$
has interesting  relations with  multinomial expansion. This  is a generalization of 
the work in arXiv:2103.06489, which built a connection between 
the Nichols algebras of squared dimension and  Pascal's triangle. 
\end{abstract}
\maketitle

\section{Introduction}
Nichols algebras appeared first in a work of Nichols \cite{MR506406}, for construction of certain pointed Hopf algebras. They also arose independently in 
Woronowicz, Lusztig, and Rosso's works \cite{woronowicz1989differential}\cite{MR1227098}\cite{MR1632802}.
Finite dimensional Nichols algebras are important ingredients for the 
classification of finite dimensional non-semisimple
Hopf algebras with dual Chevalley property \cite{andruskiewitsch2001pointed}. 

There are several equivalent definitions of Nichols algebras, but in few words one may say
that they are graded Hopf algebras in a braided category which are connected and generated
as algebras by the primitive elements which are all homogeneous of degree one.
In fact, given a braided vector space $(V, c)$ over a field $\k$, that is, a $\k$-vector space endowed with a 
solution $c\in\rm{Aut}(V\otimes V)$ of the braid equation
(also called the Yang-Baxter equation)
\[
(c\otimes \rm{id})(\rm{id}\otimes c)(c\otimes \rm{id})=
(\rm{id}\otimes c)(c\otimes \rm{id})(\rm{id}\otimes c),
\] 
we can construct a Nichols algebra $\mathfrak{B}(V,c)$(or  $\mathfrak{B}(V)$ for short).
When the braiding $c$ is  rigid,  then the Nichols algebra $\mathfrak{B}(V,c)$ can be realized in Yetter-Drinfeld categories 
of Hopf algebras \cite{Schauenburg1992} \cite{Takeuchi2000}.
For example, if $(V, c)$ is of diagonal type, then $\mathfrak{B}(V,c)$ 
can be realized in the category ${}_{\k G}^{\k G}\mathcal{YD}$ for some abelian group $G$. 
We call $(V, c)$ of group type if there is a basis $(x_i)_{i\in I}$ of $V$ and elements $g_i(x_j)\in V$
for all $i$, $j\in I$ such that 
$
c(x_i\otimes x_j)=g_i(x_j)\otimes x_i.
$
In other words, $(V, c)$ is of group type if and only if it is realizable over a group algebra $\k G$ as a 
Yetter-Drinfeld module, see \cite{Heckenberger[2020]copyright2020}. 
Given a braided vector space $(V, c)$, the determination of the dimension or the 
Gelfand-Kirillov dimension of $\mathfrak{B}(V)$ is a difficult task. 
Furthermore, once one knows the dimension of
$\mathfrak{B}(V)$, another hard problem is to present it by generators and relations.
There have been significant advances in the classification problem of finite dimensional Nichols algebras of group type.
In case the braiding $c$ is of diagonal type, the classification was completed  by Heckenberger \cite{heckenberger2009classification}
based on the theory of reflections  \cite{Heckenberger[2020]copyright2020}  and  Weyl groupoid \cite{MR2207786}, and the minimal presentation of these Nichols algebras was obtained by Angiono
\cite{Angiono2013} \cite{MR3420518}. In case $V$ is a  non-simple 
semisimple Yetter-Drinfeld module over non-abelian group
algebras, the classification were almost 
finished by  Heckenberger and  Vendramin \cite{Heckenberger2017} \cite{MR3605018}
under some technical assumptions. Only few examples are known for finite dimensional Nichols algebras over 
indecomposable braided vector spaces of group type. For these examples, please refer to 
\cite[Table 9.1]{Heckenberger2015} and the references therein. 
Heckenberger, Lochmann and Vendramin conjectured  
that any finite dimensional elementary Nichols algebra of group type 
is $bg$-equivalent to one of those listed in 
\cite[Table 9.1]{Heckenberger2015}. Besides, it is conjectured that 
non-abelian finite simple groups have no non-trivial finite dimensional Nichols algebra, see 
\cite{Andruskiewitsch2011} \cite{Andruskiewitsch2017b} \cite{Carnovale-2021} and so on.

The study for Nichols algebras of non-group type is rare. 
One possible approach to obtain finite dimensional 
Nichols algebras of non-group type is to investigate Yetter-Drinfeld categories of  Hopf 
algebras which are not categorically Morita-equivalent to group algebras, see for examples  
\cite{Xiong2019} \cite{Andruskiewitsch2020} \cite{Shi2020even}. 
The other way is to study the Nichols algebras directly from braidings \cite{Andruskiewitsch2018}
\cite{Giraldi2021}, since a Nichols algebra is  completely determined  by its braiding. 
Through  calculations of left skew derivations, 
Andruskiewitsch and Giraldi \cite[section 3.7]{Andruskiewitsch2018} found two classes of $4n$ and $n^2$-dimensional Nichols algebras, which are not of group type 
in general. The two classes of Nichols algebras can be 
realized in the  Yetter-Drinfeld categories of  
the Suzuki Hopf algebras $A_{Nn}^{\mu \lambda}$ \cite{Shi2020even} \cite{Shi2020odd}. 
In  \cite{Yuxing2020}, the author found an interesting connection between the Nichols algebras of squared dimension and the Pascal's triangle. 
In this paper, we generalize this kind of  interesting connection  to  Nichols algebras with braidings arising from  non-degenerate involutive  
solutions of the Yang-Baxter  equation and the multinomial expansion $(x_1+\cdots +x_m)^n$.

The paper is organized as follows. In the section 2, we introduce the Nichols algebras and  set-theoretical solutions of the Yang-Baxter equation.
In the section 3, we define an action of  $\Bbb{S}_n$ on $X^n$
induced by a non-degenerate involutive solution $(X, r)$ of the Yang-Baxter equation, and 
 reveal a connection  between the decomposition  of $X^n$ into orbits and 
 multinomial expansion. 
 There is a braided vector space $W_{X,r}$ induced by the action of 
 $\Bbb{S}_n$ on $X^n$. 
 We calculate the dimensions and the Gelfand–Kirillov dimensions of the Nichols algebra
 $\mathfrak{B}(W_{X,r})$, according to their connections with multinomial expansion.   
In the section $4$, we propose several questions and a conjecture for  future research. 

\section{Preliminaries}
\subsection{Nichols algebra} 
Let $\k$ be  an algebraically  closed field of characteristic $0$, and $\k^\times$ be $\k-\{0\}$.
Here we give a brief introduction to the Nichols algebra. For more details, 
please refer to Heckenberger and Schneider's monograph \cite{Heckenberger[2020]copyright2020}. 
\begin{definition}\cite[Definition 2.1]{andruskiewitsch2001pointed}%\cite[Def. 2.1]{AS} 
\label{defNicholsalgebra}
Let $H$ be a Hopf algebra and $V \in \ydh$. A braided $\mathbb{N}$-graded 
Hopf algebra $R = \bigoplus_{n\geq 0} R(n) \in \ydh$  is called 
the \textit{Nichols algebra} of $V$ if 
\begin{enumerate}
 \item[(i)] $\k \simeq R(0)$, $V\simeq R(1) \in \ydh$.
 \item[(ii)] $R(1) = \mathcal{P}(R)
=\{r\in R~|~\Delta_{R}(r)=r\otimes 1 + 1\otimes r\}$.
 \item[(iii)] $R$ is generated as an algebra by $R(1)$.
\end{enumerate}
In this case, $R$ is denoted by $\mathfrak{B}(V) = \bigoplus_{n\geq 0} \mathfrak{B}^{n}(V) $.    
\end{definition}
\begin{remark}
Let $(V, c)$ be a braided vector space, then 
the Nichols algebra 
$\mathfrak{B}(V)$ is completely determined by the braiding $c$.
More precisely, as proved in  \cite{MR1396857} and
noted in \cite{andruskiewitsch2001pointed},
$$\mathfrak{B}(V)=\k\oplus V\oplus\bigoplus\limits_{n=2}^\infty V^{\otimes n} /
\ker\mathfrak{S}_n=T(V)/\ker\mathfrak{S},$$
where $\mathfrak{S}_{n,1}\in \mathrm{End}_\k \left(V^{\otimes (n+1)}\right)$, 
$\mathfrak{S}_{n}\in \mathrm{End}_\k \left(V^{\otimes n}\right)$,
\begin{align*}
c_i&\coloneqq \mathrm{id}^{\otimes (i-1)}\otimes c
\otimes \mathrm{id}^{\otimes (n-i-1)}\in \mathrm{End}_\k \left(V^{\otimes n}\right), \\
\mathfrak{S}_{n,1} &\coloneqq\mathrm{id}+c_n+c_{n-1}c_n+\cdots
+c_1\cdots c_{n-1}c_n=\mathrm{id}+\mathfrak{S}_{n-1,1}c_n,\\
\mathfrak{S}_1&\coloneqq\mathrm{id}, \quad \mathfrak{S}_2\coloneqq\mathrm{id}+c, \quad
\mathfrak{S}_n\coloneqq \mathfrak{S}_{n-1,1}(\mathfrak{S}_{n-1}\otimes \mathrm{id}).
\end{align*}
\end{remark}

\begin{lemma}\label{TensorNicholsAlg}
(\cite[Theorem 2.2]{MR1779599}, \cite[Remark 1.4]{andruskiewitsch2021finite})
Let $M_1, M_2\in\ydh$ be both finite dimensional and assume 
$c_{M_1,M_2}c_{M_2,M_1}=\mathrm{id}_{M_2\otimes M_1}$.
Then $\mathfrak{B}(M_1\oplus M_2)\cong  \mathfrak{B}(M_1)
\otimes \mathfrak{B}(M_2)$ as graded vector spaces and 
$\mathrm{GKdim}\,\mathfrak{B}(M_1\oplus M_2)=
\mathrm{GKdim}\,\mathfrak{B}(M_1)+
\mathrm{GKdim}\,\mathfrak{B}(M_2)$, where 
$\mathrm{GKdim}$ is an abbreviation of 
the Gelfand–Kirillov dimension.
\end{lemma}

\subsection{Set-theoretical solutions of the Yang-Baxter equation}
A set-theoretical solution of the Yang-Baxter equation is a  pair $(X, r)$, where 
$X$   is a non-empty set and $r: X\times X\rightarrow X\times X$ is a bijective map
such that
\[
(r\times \rm{id})(\rm{id}\times r)(r\times \rm{id})
=(\rm{id}\times r)(r\times \rm{id})(\rm{id}\times r)
\]
holds. Here $r\times \rm{id}$ and $\rm{id}\times r$
are maps $X^3\rightarrow X^3$, $X^3=X\times X\times X$. 
By convention, we write 
\[
r(i, j)=(\sigma_i(j), \tau_j(i)),\quad  \forall i, j\in X.
\]
Then  $(X, r)$ is a set-theoretical solution of the Yang-Baxter equation if and only if the following identities  hold:
\begin{align}
\sigma_{\sigma_i(j)}\sigma_{\tau_j(i)}&=\sigma_i\sigma_j,\quad
\tau_{\tau_k(j)}\tau_{\sigma_j(k)}=\tau_k\tau_j,
\label{FormulaeYBE1}
\\
\tau_{\sigma_{\tau_j(i)}(k)}\sigma_i(j)&=\sigma_{\tau_{\sigma_j(k)}(i)}\tau_k(j), 
\quad \forall i, j, k\in X, 
\end{align}
see  Remark \ref{CoeffYBE}. 

A solution $(X, r)$ is non-degenerate if all the maps $\sigma_i: X\rightarrow X$
and $\tau_i: X\rightarrow X$ are bijective for all $i\in X$, and involutive if 
$r^2=\id_{X\times X}$.
Note that for non-degenerate involutive solutions, 
\[
\tau_j(i)=\sigma^{-1}_{\sigma_i(j)}(i),\quad 
\sigma_i(j)=\tau_{\tau_j(i)}^{-1}(j),\quad \forall i, j\in X.
\]

A solution $(X, r)$ is decomposable if $X$ is the disjoint union of $Y$ and $Z$ such that
$r(Y, Y)\subseteq Y\times Y$ and $r(Z, Z)\subseteq Z\times Z$. 
A solution $(X, r)$ is indecomposable if it is not decomposable.

If $r(i, j)=(f^{-1}(j),f(i))$ for a bijective map $f: X\rightarrow X$, then $(X, r)$
is a set-theoretical solution of the  Yang-Baxter equation. This solution $(X, r)$ is called a permutation solution, see \cite{Drinfeld1992MR1183474}
\cite{Etingof1999MR1722951}.

\begin{definition}
Let $(X, r)$ be a non-degenerate involutive solution. 
The diagonal of  the solution $(X, r)$ is defined as 
the permutation $D: X\rightarrow X$, $i\mapsto \tau_i^{-1}(i)$.
\end{definition}
\begin{remark}
$D$ is invertible with inverse 
$i\mapsto \sigma_i^{-1}(i)$ and 
\[
\tau_i^{-1}\circ D=D\circ \sigma_i
\] 
for all $i\in X$, see \cite[Proposition 2.2]{Etingof1999MR1722951}.
\end{remark}

\begin{lemma}\cite[Lemma 3.7]{10.1093/imrn/rnab232}
Let $(X, r)$ be a non-degenerate involutive solution and $i, j\in X$. Then $D(i)=j$ if and only if 
$r(j, i)=(j, i)$.
\end{lemma}

\section{Multinomial expansion and  Nichols algebras $\mathfrak{B}(W_{X, r})$}
\subsection{Multinomial expansion and the action of $\Bbb{S}_n$ on $X^n$}

\begin{definition}\label{BraidedVectorSpace}
Let $(X, r)$ be a non-degenerate   solution of the Yang-Baxter equation, $|X|=m\in\Bbb Z^{\geq 2}$.
Then $W_{X, r}=\bigoplus_{i\in X}\k w_i$ is a braided vector 
space with the braiding given by 
\begin{align}\label{YBEquation}
c(w_{i}\otimes w_j)
&=R_{i,j} w_{\sigma_i(j)}\otimes w_{\tau_j(i)}, \quad R_{i,j}\in\k^\times, \\ 
R_{i, j}R_{\tau_j(i), k}R_{\sigma_i(j),\sigma_{\tau_j(i)}(k)}
&=R_{j,k}R_{i,\sigma_j(k)}R_{\tau_{\sigma_j(k)}(i),\tau_k(j)}, 
\quad \forall i, j, k\in X. \label{YBEquation}
\end{align}
\end{definition}
\begin{remark}\label{CoeffYBE}
The formula \eqref{YBEquation} is obtained directly from the braid equation:
\begin{align*}
&\quad (c\otimes \rm{id})(\rm{id}\otimes c)(c\otimes \rm{id})(i\otimes j\otimes k)\\
&=R_{i, j}(c\otimes \rm{id})(\rm{id}\otimes c) (\sigma_i(j)\otimes \tau_j(i)\otimes k)\\
&=R_{i, j}R_{\tau_j(i), k}(c\otimes \rm{id})(\sigma_i(j)\otimes \sigma_{\tau_j(i)}(k)
\otimes \tau_k\tau_j(i))\\
&=R_{i, j}R_{\tau_j(i), k}R_{\sigma_i(j),\sigma_{\tau_j(i)}(k)}
(\sigma_{\sigma_i(j)}\sigma_{\tau_j(i)}(k)
\otimes \tau_{\sigma_{\tau_j(i)}(k)}\sigma_i(j)
\otimes \tau_k\tau_j(i)),\\
&\quad (\rm{id}\otimes c)(c\otimes \rm{id})(\rm{id}\otimes c)(i\otimes j\otimes k)\\
&=R_{j,k}(\rm{id}\otimes c)(c\otimes \rm{id})(i\otimes \sigma_j(k)\otimes \tau_k(j))\\
&=R_{j,k}R_{i,\sigma_j(k)}(\rm{id}\otimes c)
(\sigma_i\sigma_j(k)\otimes \tau_{\sigma_j(k)}(i)\otimes \tau_k(j))\\
&=R_{j,k}R_{i,\sigma_j(k)}R_{\tau_{\sigma_j(k)}(i),\tau_k(j)}
(\sigma_i\sigma_j(k)
\otimes \sigma_{\tau_{\sigma_j(k)}(i)}\tau_k(j)
\otimes \tau_{\tau_k(j)}\tau_{\sigma_j(k)}(i)).
\end{align*}

The braiding of $W_{X,r}$ is rigid according to 
\cite[Lemma 3.1.3]{Schauenburg1992}. 
\end{remark}

In the following of this section, we always assume that  $(X, r)$ is a non-degenerate  involutive  solution of the Yang-Baxter equation, $|X|=m\in\Bbb Z^{\geq 2}$. 
Let $\mathbb{S}_n$ be the symmetric group on $n$ letters. 
The  symmetric group $\Bbb{S}_n$ is generated by transpositions $s_1$, $\cdots$, $s_{n-1}$.
Denote $X^n=\{i_1i_2\cdots i_n\mid i_j\in X, 1\leq j\leq n\}$. 
There is an action of $\Bbb{S}_n$ on $X^n$ induced by the solution $(X, r)$ such that 
\[
 s_k\cdot \left(i_1\cdots i_{k-1}pqi_{k+2}\cdots i_n\right)
=i_1\cdots i_{k-1}p^\prime q^\prime i_{k+2}\cdots i_n,\quad 
r(p,q)=(p^\prime, q^\prime).
\]
Let $\mathcal{O}(x)=\Bbb{S}_n\cdot x$ be the orbit of $x\in X^n$. 
Denote $\mathfrak{G}(x,y)
=\{\sigma\in\Bbb{S}_n\mid \sigma \cdot x=y\}$ for $x$, $y\in X^n$.

%\begin{definition}[Lyubashenko; see \cite{Etingof1999MR1722951}]
%Let $X$ be a non-empty set  and $f: X\rightarrow X$ be a bijective map. Then $s(x, y)=(f^{-1}(y),f(x))$ 
%is a set-theoretical solution of the  Yang-Baxter equation. This solution $(X, s)$ is called a permutation 
%solution. 
%\end{definition}
%Let $m\in \Bbb{Z}$, $m\geq 2$,  $X^n=\{i_1i_2\cdots i_n\mid i_j\in X, 1\leq j\leq n\}$. 
%Let $\tau$ be a permutation which acts on $\Bbb{Z}_m$, then 
%$s(i, j)=(\tau^{-1}(j), \tau(i))$ is a permutation solution. 
%There is  an action of $\Bbb{S}_n$ on $\Bbb{Z}_m^n$ induced from the 
%permutation solution $(\Bbb Z_m, s)$ as follows:
%\begin{align*}
%&\quad s_k\cdot \left(i_1\cdots i_{k-1}pq i_{k+2}\cdots i_n\right)
%=i_1\cdots i_{k-1}p^\prime q^\prime i_{k+2}\cdots i_n, \quad
%p^\prime=\tau^{-1}(q),\quad q^\prime=\tau(p).
%\end{align*}

As for $x=i_1i_2\cdots i_n\in X^n$, we denote $w_x=w_{i_1}w_{i_2}\cdots w_{i_n}$
for abbreviation of $w_{i_1}\otimes w_{i_2}\otimes \cdots \otimes w_{i_n}$.
Define 
$\mathcal{T}: \Bbb{S}_n\rightarrow \mathrm{End}_{\k}(W_{X,r}^{\otimes n})$ such that
$\mathcal{T}_{s_i}=c_i$
and if $\theta=s_{j_1}s_{j_2}\cdots s_{j_t}\in \Bbb{S}_n$ is a reduced expression, then 
$\mathcal{T}_\theta=c_{j_1}c_{j_2}\cdots c_{j_t}$. We have
$$\mathfrak{S}_n(w_x)=\sum_{\theta\in\Bbb{S}_n}\mathcal{T}_{\theta}(w_x),\quad \forall x\in X^n.$$

For $k_1+\cdots+k_r=n$, the set of $(k_1,\cdots, k_r)$-shuffles, i.e. the set of permutations $w$
such that $w(1)<w(2)<\cdots<w(k_1)$, 
$w(k_1+1)<w(k_1+2)<\cdots<w(k_1+k_2)$, 
$\cdots$, $w(k_1+\cdots+k_{r-1}+1)<\cdots<w(n)$, 
is denoted by $\mathrm{shuffle}(k_1,\cdots,k_r)$.
Let  $\Bbb{G}_n$ be the set of $n$-th primitive roots of unity. 
Denote by $(n)_q=1+q+\cdots +q^{n-1}$,
 $(n)_q^!=\prod_{k=1}^n(k)_q$.

Denote $\mathcal{P}(n,m)$ the set of integer partitions  
$\lambda=(\lambda_1, \lambda_2,\cdots, \lambda_m)$  such that 
$\lambda_1\geq \lambda_2\geq \cdots \geq \lambda_m\geq 0$, $\lambda_1+\lambda_2+\cdots 
+\lambda_m=n$. In general,  we omit the zero parts of $\lambda$. 
The number of permutations of $\lambda=(\lambda_1, \lambda_2,\cdots, \lambda_m)\in \mathcal{P}(n,m)$
is denoted as $\mathrm{Perm(\lambda)}$. 
A partition $\lambda\in \mathcal{P}(n,m)$ can be described as 
$n^{k_n}, (n-1)^{k_{n-1}},\cdots, 1^{k_1}, 0^{k_0}$, where $k_i$ is the number of parts of $\lambda$
equal to $i$. According to the permutation formula with repetition, we have 
$$\mathrm{Perm(\lambda)}=\frac{m!}{k_n!k_{n-1}!\cdots k_1!k_0!}.$$
For example,  $\lambda=(3,2,2,2,0,0)\in \mathcal{P}(9,6)$, 
 $\mathrm{Perm(\lambda)}=\frac{6!}{3!2!}=60$. 

\begin{lemma}\label{MultinomialPartition}
Let $m$, $n\in \Bbb Z^+$, then 
$m^n
%=\sum\limits_{a_1+a_2+\cdots+a_m=n}
%\frac{n!}{a_1! a_2!\cdots a_m!}
=\sum\limits_{\lambda \in \mathcal{P}(n,m)}
\frac{n!}{\lambda_1!\lambda_2!\cdots \lambda_m!}\cdot \mathrm{Perm(\lambda)}$.
\end{lemma}
\begin{proof}
The multinomial expansion 
\begin{align*}
(x_1+x_2+\cdots+x_m)^n
=\sum_{a_1+a_2+\cdots+a_m=n}\frac{n!}{a_1!a_2!\cdots  a_m!}
x_1^{a_1}x_2^{a_2}\cdots x_m^{a_m}
\end{align*}
implies that  $m^n=\sum\limits_{a_1+a_2+\cdots+a_m=n}
\frac{n!}{a_1! a_2!\cdots a_m!}$.
Now the conclusion can be drawn from
\begin{align*}
&\quad\,\left\{(a_1,a_2,\cdots, a_m)\in\Bbb N^m\mid a_1+a_2+\cdots +a_m=n\right\}\\
&= \left\{\Bbb S_m\cdot \lambda\mid \lambda \in \mathcal{P}(n,m), 
\sigma\cdot \lambda=(\lambda_{\sigma(1)}, \lambda_{\sigma(2)},
\cdots, \lambda_{\sigma(m)}), \forall \sigma\in \Bbb S_m\right\}.
\end{align*}
\end{proof}

Let $x=a_1a_2\cdots a_k\in  X^k$, define
$
\sigma_x=\sigma_{a_1}\sigma_{a_2}\cdots\sigma_{a_k}$,
$\tau_x=\tau_{a_k}\cdots\tau_{a_2}\tau_{a_1}.
$
Denote 
\[
\Psi_k(a)=D^{k-1}(a)D^{k-2}(a)\cdots D(a)a=\Psi_{-k}(D^{k-1}(a))\in X^k,\quad 
\forall a\in X.
\]

%\begin{lemma}[Exchange rule]
%Let $a=a_1a_2\cdots a_p\in  X^p$, $b=b_1b_2\cdots b_q\in  X^q$, 
%$\sigma_a=\sigma_{a_1}\sigma_{a_2}\cdots\sigma_{a_r}$,  and 
%$\tau_b=\tau_{b_q}\cdots\tau_{b_2}\tau_{b_1}$.
%Then 
%\[
%\mathcal{O}(a_1\cdots a_pb_1\cdots b_q)
%=\mathcal{O}(b_1^\prime\cdots b_q^\prime a_1^\prime\cdots a_p^\prime), 
%\]
%with $b_i^\prime=\sigma_a(b_i)$  and $a_j^\prime=\tau_{b}(a_j)$. 
%\end{lemma}
%
%
%Denote $\Psi_n(a)= i_1i_2\cdots i_n\in X^n$ such that $i_{j}=D(i_{j+1}) \in \Bbb{Z}_m$
%for $j=1,\cdots, n-1$, $i_n=a$. According to the exchange rule, we have 
%\begin{align*}
%\mathcal{O}\left(\Psi_n(a)\Psi_m(b)\right)
%=\mathcal{O}\left(\sigma_{\Psi_n(a)}(D^{m-1}(b))\cdots  
%   \sigma_{\Psi_n(a)}(b) 
%   \tau_{\Psi_m(b)}(D^{n-1}(a))\cdots \tau_{\Psi_m(b)}(a)\right)
%\end{align*}

\begin{lemma}\label{ExchangRulem=1}
For any $k\in\Bbb Z^+$, $x, y\in X$, we have 
\[
\mathcal{O}\left(\Psi_k(x)y\right)
=\mathcal{O}\left(\sigma_{\Psi_k(x)}(y)\Psi_k(\tau_{y}(x))\right).
\]
\end{lemma}
\begin{proof}
\begin{figure}[h!]
\begin{tikzpicture}
\draw[->](8,2)--(0,0);
\draw[->](6,2)--(8,0);
\draw[->](4,2)--(6,0);
\draw[->](2,2)--(4,0);
\draw[->](0, 2)--(2, 0);
\node[above] at (8,2) {$y$}; 
\node[above] at (6,2) {$x_1$}; 
\node[above] at (4,2) {$x_2$};
\node[above] at (2, 2) {$x_{k-1}$};
\node[above] at (0, 2) {$x_k$};
\node[below] at (0,0) {$z_k$}; 
\node[below] at (2,0) {$a_k$};
\node[below] at (4,0) {$a_{k-1}$};
\node[below] at (6, 0) {$a_2$};
\node[below] at (8, 0) {$a_1$};
\node at (2.2, 0.8) {$z_{k-1}$};
\node at (3, 2.2) {$\cdots$};
\node at (5, -0.2) {$\cdots$};
\node at (6, 1.2) {$z_1$};
\end{tikzpicture}
\caption{Exchange rule for the case $t=1$, $\forall k\in \Bbb Z^+$}
\label{Exchange rulem=1}
\end{figure}
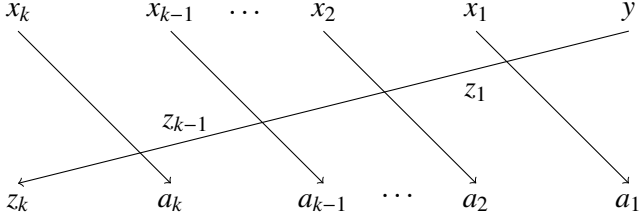
%%%%%%%%%%%%%%%%%%%%%%%%%%
Suppose $\Psi_k(x)=x_kx_{k-1}\cdots x_2x_1\in X^k$, 
$r(x_1, y)=(z_1, a_1)$, $r(x_{i+1}, z_i)=(z_{i+1}, a_i)$ for $i=1, \cdots, k-1$, see Figure \ref{Exchange rulem=1}. 
As for $1\leq i\leq k-1$, we have
\[
\tau_{z_i}^{-1}D(a_i)=D\sigma_{z_i}(a_i)=D(x_i)=x_{i+1}
\Rightarrow D(a_i)=\tau_{z_i}(x_{i+1})=a_{i+1}.
\]
It is easy to see that $a_1=\tau_y(x_1)=\tau_y(x)$, and 
\begin{align*}
z_k&=\sigma_{x_k}(z_{k-1})
=\sigma_{x_k}\sigma_{x_{k-1}}(z_{k-2})
=\cdots
=\sigma_{x_k}\cdots \sigma_{x_3}\sigma_{x_2}(z_1)\\
&=\sigma_{x_k}\cdots \sigma_{x_2}\sigma_{x_1}(y)
=\sigma_{\Psi_k(x)}(y).
\end{align*}
So $\mathcal{O}\left(\Psi_k(x)y\right)
=\mathcal{O}(z_ka_ka_{k-1}\cdots a_1)
=\mathcal{O}\left(\sigma_{\Psi_k(x)}(y)\Psi_k(\tau_{y}(x))\right)$.
\end{proof}

\begin{lemma}[Exchange Rule]
For any $x, y\in X$, $t, k\in \Bbb Z^+$, we have 
$$\mathcal{O}\left(\Psi_k(x)\Psi_t(y)\right)
=\mathcal{O}\left(
\Psi_{-t}\left(\sigma_{\Psi_k(x)}D^{t-1}(y)\right)\Psi_k(\tau_{\Psi_t(y)}(x))
\right).$$
\end{lemma}
\begin{proof}
According to Lemma \ref{ExchangRulem=1}, the formula holds for 
any $k\in \Bbb Z^+$ and $t=1$. 
Suppose the formula  holds for  $\mathcal{O}\left(\Psi_k(x)\Psi_{t-1}(y)\right)$, 
we prove the formula holds for $\mathcal{O}\left(\Psi_k(x)\Psi_t(y)\right)$ 
by induction. 
\begin{align*}
&\quad\mathcal{O}\left(\Psi_k(x)\Psi_t(y)\right)\\
&=\mathcal{O}\left(\Psi_k(x)\Psi_{t-1}(D(y))y\right)\\
&=\mathcal{O}\left(\Psi_{-(t-1)}\left(\sigma_{\Psi_k(x)}D^{t-2}(D(y))\right)\Psi_k(\tau_{\Psi_{t-1}(D(y))}(x))y\right)\\
&=\mathcal{O}\left(\Psi_{-(t-1)}\left(\sigma_{\Psi_k(x)}D^{t-1}(y)\right)
\sigma_{\Psi_k(\tau_{\Psi_{t-1}(D(y))}(x))}(y)\Psi_k(\tau_y\tau_{\Psi_{t-1}(D(y))}(x))\right)\\
&=\mathcal{O}\left(\Psi_{-(t-1)}\left(\sigma_{\Psi_k(x)}D^{t-1}(y)\right)
\sigma_{\Psi_k(\tau_{\Psi_{t-1}(D(y))}(x))}(y)\Psi_k(\tau_{\Psi_{t}(y)}(x))\right).
\end{align*}
As depicted in Figure \ref{Exchange rulemn}, we denote 
\begin{align*}
\Psi_k(x)&=x_kx_{k-1}\cdots x_1, &
\Psi_k(\tau_{\Psi_{t}(y)}(x))&=a_ka_{k-1}\cdots a_1,\\
\Psi_t(y)&=y_ty_{t-1}\cdots y_1,&
b_tb_{t-1}\cdots b_2&=\Psi_{-(t-1)}\left(\sigma_{\Psi_k(x)}D^{t-1}(y)\right), \\
b_1&=\sigma_{\Psi_k(\tau_{\Psi_{t-1}(D(y))}(x))}(y),&
z_kz_{k-1}\cdots z_1&=\Psi_k(\tau_{\Psi_{t-1}(D(y))}(x)).
\end{align*}
\begin{figure}[h!]
\begin{tikzpicture}
\draw[->](10.8,3)--(5.4,0);
\draw[->](9,3)--(3.6,0);
\draw[->](7.2,3)--(1.8,0);
\draw[->](5.4,3)--(0,0);
\draw[->](3.6,3)--(10.8,0);
\draw[->](1.8,3)--(9,0);
\draw[->](0, 3)--(7.2, 0);
\node[above] at (10.8,3) {$y_1$}; 
\node[above] at (9,3) {$y_2$}; 
\node[above] at (7.2,3) {$y_3$}; 
\node[above] at (5.4,3) {$y_t$}; 
\node[above] at (3.6,3) {$x_1$};
\node[above] at (1.8, 3) {$x_{k-1}$};
\node[above] at (0, 3) {$x_k$};
\node[below] at (0,0) {$b_t$}; 
\node[below] at (1.8,0) {$b_3$};
\node[below] at (3.6,0) {$b_2$};
\node[below] at (5.4, 0) {$b_1$};
\node[below] at (7.2, 0) {$a_k$};
\node[below] at (9.3, 0) {$a_{k-1}$};
\node[below] at (10.8, 0) {$a_1$};
\node at (2.8, 3.2) {$\cdots$};
\node at (6.3, 3.2) {$\cdots$};
\node at (0.8, -0.2) {$\cdots$};
\node at (10.1, -0.2) {$\cdots$};
\node at (5.8, 0.75) {$z_k$};
\node at (6.6, 1.1) {$z_{k-1}$};
\node at (7.0, 1.3) {$\cdots$};
\node at (7.4, 1.6) {$z_{1}$};
\end{tikzpicture}
\caption{Exchange rule for  $\forall t,  k\in \Bbb Z^+$}
\label{Exchange rulemn}
\end{figure}
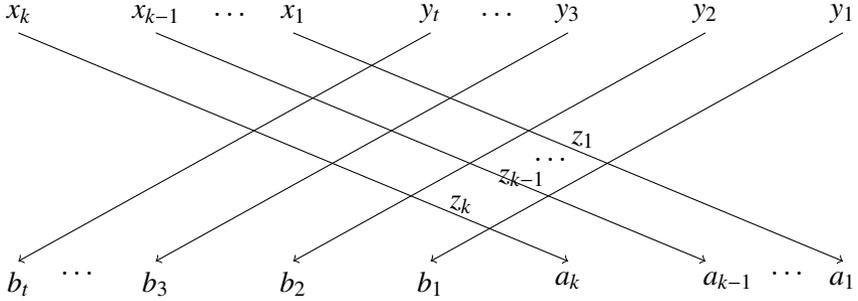
%%%%%%%%%%%%%%%%%%%%%%%%%%
Since 
$
\tau_{\Psi_{t-1}(D(y))}(x)
%&=\tau_{D^{m-1}(y)D^{m-2}(y)\cdots D(y)}(x)
=\tau_{y_ty_{t-1}\cdots y_2}(x)
=\tau_{y_2}\cdots \tau_{y_{t-1}}\tau_{y_t}(x)=z_1
$,  we have 
%\end{align*}
\begin{align*}
D(b_1)=D\sigma_{\Psi_{k}(z_1)}(y)
&=D\sigma_{z_k}\sigma_{z_{k-1}}\cdots \sigma_{z_1}(y)
=\tau^{-1}_{z_k}\tau_{z_{k-1}}^{-1}\cdots \tau_{z_1}^{-1}D(y),\\
\tau_{z_1}\cdots\tau_{z_{k-1}}\tau_{z_k}(b_2)
&=y_2=D(y), \quad \text{(see Figure \ref{Exchange rulemn})}. 
\end{align*}
We obtain $D(b_1)=b_2$. The proof is finished. 
\end{proof}

\begin{definition}
If $x=\Psi_{\lambda_1}(a_1)\Psi_{\lambda_2}(a_2)\cdots \Psi_{\lambda_k}(a_k)
$ 
such that $\lambda=(\lambda_1,\cdots,\lambda_k)\in \mathcal{P}(n,m)$, 
$\lambda_1\geq \lambda_2\geq\cdots\geq \lambda_k>0$, 
and 
\begin{align}
a_j\neq D^{-\lambda_j}\tau_{\Psi_{\lambda_{i+1}}(a_{i+1})
\cdots \Psi_{\lambda_{j-1}}(a_{j-1})}(a_i), \quad 
1\leq i<j\leq k, 
\label{eqk}
\end{align}
then we say $x$ is a $\lambda$-element and the orbit $\mathcal{O}(x)$ corresponds to the partition $\lambda$. 
And we let $\mathcal{B}(\lambda)$ be the union of all orbits corresponding to 
the partition $\lambda$.
\end{definition}
\begin{remark}\label{DefinitionL_Elemenet}
If $i<k<j$ such that 
$$D^{-\lambda_j}\tau_{\Psi_{\lambda_{i+1}}(a_{i+1})
\cdots \Psi_{\lambda_{j-1}}(a_{j-1})}(a_i)
=D^{-\lambda_j}\tau_{\Psi_{\lambda_{k+1}}(a_{k+1})
\cdots \Psi_{\lambda_{j-1}}(a_{j-1})}(a_k), 
$$ 
then 
we have $\tau_{\Psi_{\lambda_{i+1}}(a_{i+1})
\cdots \Psi_{\lambda_{k-1}}(a_{k-1})}(a_i)=\left(\tau_{\Psi_{\lambda_k}(a_k)}\right)
^{-1}(a_k)=D^{\lambda_k}(a_k)$. 
It is a contradiction with that $x$ is a $\lambda$-element. 
So there are $m-i$ choices for $a_{i+1}$, 
which implies that the number of $\lambda$-elements is  $\frac{m!}{(m-k)!}$.
\end{remark}

\begin{lemma}\label{ExchangePartsOfSameLength}
Suppose $\lambda=(\lambda_1,\cdots, \lambda_k)\in\mathcal{P}(n,m)$ with $\lambda_i=\lambda_{i+1}>0$ 
for some $i\in\{1,2,\cdots,k-1\}$. 
Let  $x=\Psi_{\lambda_1}(a_1)\Psi_{\lambda_2}(a_2)\cdots \Psi_{\lambda_k}(a_k)$
be a $\lambda$-element, and  
$x^\prime$ be the element obtained from $x$ by exchanging positions  of 
$\Psi_{\lambda_i}(a_i)$ and $\Psi_{\lambda_{i+1}}(a_{i+1})$ under the exchange rule.  
Then $x^\prime$ is still a $\lambda$-element. 
\end{lemma}
\begin{proof}
It is a direct verification.
\end{proof}

\begin{lemma}\label{LemmaL_Elemenet}
If $x=\Psi_{\lambda_1}(a_1)\Psi_{\lambda_2}(a_2)\cdots \Psi_{\lambda_k}(a_k)$
is a $\lambda$-element for the partition $\lambda=(\lambda_1,\cdots,\lambda_k)\in\mathcal{P}(n,m)$ with $\lambda_1=\lambda_2=\cdots=\lambda_k$, 
then the number of  $\lambda$-elements in $\mathcal{O}(x)$ is $k!$.
\end{lemma}
\begin{proof}
If $\lambda=(\lambda_1)=(n)\in \mathcal{P}(n,m)$, then there is exactly one $\lambda$-element in $\mathcal{O}(x)$.
Suppose the conclusion holds for partitions of $\mathcal{P}(n,m)$ with $k$ non-zero parts and 
the $k$ parts are the same. 

Let $x_{k+1}=\Psi_{\lambda_1}(a_1)\Psi_{\lambda_1}(a_2)\cdots \Psi_{\lambda_1}(a_k)
\Psi_{\lambda_1}(a_{k+1})$
be a $\lambda$-element for the partition $\lambda=(\lambda_1,\cdots,\lambda_1)\in\mathcal{P}(n,m)$.
Then 
$$a_j\neq D^{-\lambda_1}\tau_{\Psi_{\lambda_{1}}(a_{i+1})
\cdots \Psi_{\lambda_{1}}(a_{j-1})}(a_i)$$
 for $1\leq i<j\leq k+1$. 
According  to the exchange rule, if we move the term $\Psi_{\lambda_1}(a_l)$
to the tail of $x_{k+1}$ for $l=2,\cdots, k+1$,  then we obtain 
$$
x_l=\cdots \Psi_{\lambda_1}\left(
\tau_{\Psi_{\lambda_1}(a_{l+1})\Psi_{\lambda_1}(a_{l+2})\cdots\Psi_{\lambda_1}(a_{k+1})}
(a_{l})\right)\in\mathcal{O}(x_{k+1}).
$$
Those $x_l$ are $\lambda$-elements by  Lemma \ref{ExchangePartsOfSameLength}. 
Suppose $x_{l_1}=x_{l_2}$ for $l_1<l_2$,  then 
\[
\tau_{\Psi_{\lambda_1}(a_{l_1+1})\Psi_{\lambda_1}(a_{l_1+2})\cdots\Psi_{\lambda_1}(a_{k+1})}(a_{l_1})
=\tau_{\Psi_{\lambda_1}(a_{l_2+1})\Psi_{\lambda_1}(a_{l_2+2})\cdots\Psi_{\lambda_1}(a_{k+1})}(a_{l_2}), 
\]
which implies that 
$$D^{\lambda_1}(a_{l_2})
=\left(\tau_{\Psi_{\lambda_1}(a_{l_2})}\right)^{-1}(a_{l_2})
=\tau_{\Psi_{\lambda_1}(a_{l_1+1})\Psi_{\lambda_1}(a_{l_1+2})\cdots\Psi_{\lambda_1}(a_{l_2-1})}(a_{l_1}).
$$
It is a contradiction.  
In other words, the tails of $\lambda$-elements $x_l$ for $l=1,\cdots, k+1$ 
are different, so the number of $\lambda$-elements in $\mathcal{O}(x_{k+1})$ is  $(k+1)!$ by induction. 
\end{proof}

\begin{lemma}\label{NumberObits}
Let $\lambda\in \mathcal{P}(n,m)$, 
then the number of orbits in $\mathcal{B}(\lambda)$ is  $\rm{Perm}(\lambda)$.
\end{lemma}
\begin{proof}
The partition $\lambda\in \mathcal{P}(n,m)$  can be described as 
$n^{k_n}, (n-1)^{k_{n-1}},\cdots, 1^{k_1}, 0^{k_0}$, where $k_i$ is the number of parts of $\lambda$
equal to $i$. Suppose $x$ is any $\lambda$-element, then the number of $\lambda$-elements in 
$\mathcal{O}(x)$ is $k_n!k_{n-1}!\cdots k_2!k_1!$ according to  Lemma 
\ref{ExchangePartsOfSameLength} and \ref{LemmaL_Elemenet}. 
The total number of  $\lambda$-elements is $\frac{m!}{k_0!}$ according to  Remark \ref{DefinitionL_Elemenet}. So the number of orbits in $\mathcal{B}(\lambda)$ is
\[
\frac{\frac{m!}{k_0!}}{k_n!k_{n-1}!\cdots k_2!k_1!}=\rm{Perm}(\lambda).
\]
\end{proof}

\begin{lemma}\label{AllElement}
For any $y\in X^n$, there exists a partition 
$\lambda\in\mathcal{P}(n,m)$ and a $\lambda$-element $x$
such that $y\in\mathcal{O}(x)$. As a consequence,  
$X^n=\bigcup_{\lambda \in \mathcal{P}(n,m)} \mathcal{B}(\lambda)$.
\end{lemma}
\begin{proof}
A composition of $n$ is a sequence $\mu=(\mu_1,\cdots,\mu_k)$ such that 
$\mu_i\in\Bbb{Z}^{+}$
for $i=1,\cdots, k$ and $\mu_1+\cdots+\mu_k=n$. 
The length of a composition is the number of its parts, denoted as $l(\mu)$. 
For any $y\in X^n$, we can write $y$
in the unique form $\Psi_{\mu_1}(a_1)\cdots \Psi_{\mu_k}(a_k)$
for a composition $\mu$ of $n$  such that 
$a_{j-1}\neq D^{\lambda_j}(a_j)$, $j=2,\cdots,k$.
As for convenience, we define  $l(y)=l(\mu)$. 
We prove $y\in\mathcal{O}(x)$ by induction on $l(y)$, where $x$ is a $\lambda$-element for 
some $\lambda\in \mathcal{P}(n, m)$.
\par 
If $l(y)=1$, then $y$ is already a $\lambda$-element for $\lambda=(n)\in \mathcal{P}(n, m)$. 
Suppose the conclusion holds for the case $l(y)\leq k-1$,  we prove it also holds for the case $l(y)=k$ in the following. 
\par 
We can rearrange $k$ parts of $y$ in an ordered way according to the exchange rule. That is to say, 
there exists a 
\[
y^\prime=\Psi_{\mu^\prime_1}(\alpha_1)\cdots \Psi_{\mu^\prime_{k}}(\alpha_{k})
\in\mathcal{O}(y), 
\]
where $(\mu^\prime_1,\cdots, \mu^\prime_k)$ is a permutation of $(\mu_1,\cdots, \mu_k)$ such 
that $\mu^\prime_1\geq \mu^\prime_2\geq \cdots\geq \mu^\prime_k$. 
\par
If $\alpha_j=D^{-\mu^\prime_j}\tau_{\Psi_{\mu^\prime_{i+1}}(\alpha_{i+1})
\cdots \Psi_{\mu^\prime_{j-1}}(\alpha_{j-1})}(\alpha_i)$ for some $i$ and $j$ with $1\leq i<j\leq k$, 
then we can  move $\Psi_{\mu_j^\prime}(\alpha_j)$ forward to join with 
$\Psi_{\mu_i^\prime}(\alpha_i)$ according to the exchange rule. 
In other words, there exists a $y^{\prime\prime}\in  X^n$, such that 
$y^{\prime\prime}\in \mathcal{O}(y^\prime)$
and $l(y^{\prime\prime})<l(y^\prime)\leq k$. 
By induction, $y^{\prime\prime}\in\mathcal{O}(x)\subseteq \mathcal B(\lambda)$ for some $\lambda\in\mathcal
P(n,m)$ and a $\lambda$-element $x$.  So $y\in\mathcal{O}(x)$. 
\par 
If $\alpha_j\neq D^{-\mu^\prime_j}\tau_{\Psi_{\mu^\prime_{i+1}}(\alpha_{i+1})
\cdots \Psi_{\mu^\prime_{j-1}}(\alpha_{j-1})}(\alpha_i)$ holds for $1\leq i<j\leq k$. Then 
%\begin{equation}
$$
\alpha_k\neq D^{-\mu^\prime_k}\tau_{\Psi_{\mu^\prime_{i+1}}(\alpha_{i+1})
\cdots \Psi_{\mu^\prime_{k-1}}(\alpha_{k-1})}(\alpha_i)
\in X, \quad 
1\leq i\leq k-1.
$$
%\end{equation}
If $k>m$, then there exist some $i$ and $j$ with $i<j$, such that 
\[
D^{-\mu^\prime_k}\tau_{\Psi_{\mu^\prime_{i+1}}(\alpha_{i+1})
\cdots \Psi_{\mu^\prime_{k-1}}(\alpha_{k-1})}(\alpha_i)
=D^{-\mu^\prime_k}\tau_{\Psi_{\mu^\prime_{j+1}}(\alpha_{j+1})
\cdots \Psi_{\mu^\prime_{k-1}}(\alpha_{k-1})}(\alpha_j)
\in X, 
\]
since $\alpha_k\in X$ with $|X|=m$. So 
$$\tau_{\Psi_{\mu^\prime_{i+1}}(a_{i+1})
\cdots \Psi_{\mu^\prime_{j-1}}(a_{j-1})}(a_i)=\left(\tau_{\Psi_{\mu^\prime_j}(a_j)}\right)
^{-1}(a_j)=D^{\mu^\prime_j}(a_j).$$
It is a contradiction. 
So $k\leq m$. 
Now we see $y^\prime$ is a $\mu^\prime$-element, where
$\mu^\prime=(\mu^\prime_1,\cdots, \mu^\prime_k)\in\mathcal{P}(n,m)$. 
\end{proof}
\begin{example}
As for  $0121212020102\in X^{13}$ with  $X=\Bbb{Z}_3$, $r(i,j)=(j-1,i+1)\in X\times X$, 
$i, j\in X$, we have
\begin{align*}
\mathcal{O}(0121212020102)
&=\mathcal{O}\left(012\,{\color{blue}12\,120\,201}\,{\color{green}0}\,2\right)\\
&=\mathcal{O}\left(012\,{\color{green}1}\, {\color{blue}20\,201\,012}\,2\right)
   =\mathcal{O}\left(012\,120\,201\,012\,2\right)\\
&=\mathcal{O}\left(012\,120\,{\color{blue}201\,012}\,{\color{green}2}\right)\\
&=\mathcal{O}\left(012\,120\,{\color{green}2}\,{\color{blue}012\,120}\right)
   =\mathcal{O}\left(012\,120\,2012\,120\right)\\
&=\mathcal{O}\left({\color{blue}012\,120}\,{\color{green}2012}\,120\right)
   =\mathcal{O}\left({\color{green}2012}\,{\color{blue}120\,201}\,120\right)\\
&=\mathcal{O}\left(2012\,{\color{red}120\,201}\,{\color{blue}120}\right)\\
&=\mathcal{O}\left(2012\,{\color{blue}120}\,{\color{red}120\,201}\right)
=\mathcal{O}\left(2012\,120120\,201\right)\\
&=\mathcal{O}\left({\color{blue}2012}\,{\color{red}120120}\,201\right)
=\mathcal{O}\left({\color{red}012012}\,{\color{blue}2012}\,201\right).
\end{align*}
012012\,2012\,201 is a $(6,4,3)$-element. 
\end{example}

%\begin{lemma}\label{GroupAction}
%Let $G$ be a finite group that acts on a set $X$. Denote 
%\[
%G_x=\{g\in G\mid g\cdot x=x\},\qquad 
%\mathcal{O}(x)=G\cdot x,\quad x\in X. 
%\]
%Then $|G|=|G_x|\cdot |\mathcal{O}(x)|$.
%\end{lemma}

\begin{theorem}\label{MultinomialConjecture}
Let  $\lambda=(\lambda_1,\cdots, \lambda_k)\in \mathcal{P}(n,m)$. 
If $x$ is a $\lambda$-element,  then 
\[
\mathfrak{G}(x,x)=\Bbb{S}_{\lambda_1}\times 
\cdots\times \Bbb{S}_{\lambda_k}, \quad 
|\mathcal{O}(x)|=\frac{n!}{\lambda_1!\lambda_2!\cdots \lambda_k!}.
\] 
The number of orbits in $X^n$ is ${n+m-1\choose m-1}$.
\end{theorem}
\begin{proof}
Since $x$ is  a  $\lambda$-element,  $(\Bbb{S}_{\lambda_1}\times 
\cdots\times \Bbb{S}_{\lambda_k})\cdot x=x$.  We have 
$|\mathcal{O}(x)|\leq \frac{n!}{\lambda_1!\lambda_2!\cdots \lambda_m!}$ by 
 the orbit-stabilizer theorem.

According to  Lemma \ref{AllElement}, \ref{NumberObits} and \ref{MultinomialPartition}, we have 
\begin{align*}
m^n&=|X^n|=\sum_{\lambda\in\mathcal{P}(n,m)}|\mathcal{B}(\lambda)|
=\sum_{\lambda\in\mathcal{P}(n,m), \mathcal{O}(x)\,\text{corresponding to}\,\lambda} 
     |\mathcal{O}(x)|\cdot \rm{Perm}(\lambda)\\
&= \sum_{\lambda \in \mathcal{P}(n,m)}
\frac{n!}{\lambda_1!\lambda_2!\cdots \lambda_m!}\cdot \mathrm{Perm(\lambda)}.
\end{align*}
So $|\mathcal{O}(x)|= \frac{n!}{\lambda_1!\lambda_2!\cdots \lambda_k!}$ follows from 
the above identity. 
According to  the orbit-stabilizer theorem, $\mathfrak{G}(x,x)=\Bbb{S}_{\lambda_1}\times 
\cdots\times \Bbb{S}_{\lambda_k}$.

The number of orbits in $X^n$ is exactly the number of terms in the  
expansion of $(x_1+x_2+\cdots+x_m)^n$. So it equals to ${n+m-1\choose m-1}$.
\end{proof}
\begin{corollary}
\[
\sum_{\lambda\in\mathcal{P}(n,m)}\rm{Perm}(\lambda)={n+m-1\choose m-1}.
\]
\end{corollary}
\begin{remark}
Let $M=\oplus_{y\in \mathcal{O}(x)}\k y$ be a $\Bbb{S}_n$-module induced by 
$ \mathcal{O}(x)$, then 
\[M \cong  \mathrm{Ind}_{\Bbb{S}_{\lambda_1}\times 
\cdots\times \Bbb{S}_{\lambda_k}}^{\Bbb S_n}\k \cong \mathrm{M}^\lambda, 
\]
where 
$\mathrm{M}^\lambda$ is the permutation module 
of $\Bbb{S}_n$ on $\lambda$-tabloids, see \cite{Fulton1997}. 
By abuse of notation, we also denote  $\oplus_{y\in \mathcal{O}(x)}\k y$ by $\mathcal{O}(x)$. 
%So $|\mathcal{O}_x|=\frac{n!}{\lambda_1!\lambda_2!\cdots \lambda_k!}
%=\frac{n!}{\lambda_1!\lambda_2!\cdots \lambda_m!}$.  
\end{remark}
\begin{example}
Let $(\Bbb Z_3, r)$ with
$r(i, j)=(j-1, i+1)\in\Bbb Z_3\times \Bbb Z_3$  be a permutation solution,  
then $\Bbb{Z}_3^4$ is the union of the following orbits. 
\begin{align*}
\mathcal{O}(0120)&=\{0120\}\cong \mathrm{M}^{(4)},  \\
\mathcal{O}(1201)&=\{1201\}\cong \mathrm{M}^{(4)},  \quad
\mathcal{O}(2012)=\{2012\}\cong \mathrm{M}^{(4)},  \\
\mathcal{O}(0122)&=\{0122, 0110, 0020, 2120\}\cong \mathrm{M}^{(3,1)},\\
\mathcal{O}(0121)&=\{0121, 0100, 0220, 1120\}\cong \mathrm{M}^{(3,1)},\\
\mathcal{O}(1200)&\cong \mathrm{M}^{(3,1)},\quad 
\mathcal{O}(1202)\cong \mathrm{M}^{(3,1)},\quad
\mathcal{O}(2011)\cong \mathrm{M}^{(3,1)},\quad 
\mathcal{O}(2010)\cong \mathrm{M}^{(3,1)},\\
\mathcal{O}(0101)&=\{0101, 0221, 1121, 0200, 1100, 1220\}=\mathcal{O}(1220) \cong \mathrm{M}^{(2,2)},\\
\mathcal{O}(0112)&=\{0112, 0022, 2122, 0010, 2110, 2020\}=\mathcal{O}(2020)\cong \mathrm{M}^{(2,2)},\\
\mathcal{O}(1212)&=\{1212, 1002, 2202, 1011, 2211, 2001\}=\mathcal{O}(2001)\cong \mathrm{M}^{(2,2)},\\
\mathcal{O}(0102)&=\{0210, 1110, 0000, 0222, 1020, 1122, 2100, 0021, 0102, 2220, \\
&\quad\,\,\, 2121, 0111\}=\mathcal{O}(0111)\cong\mathrm{M}^{(2,1,1)},\\
\mathcal{O}(2000)&=\{2102, 0002, 2222, 2111, 0212, 0011, 1022, 2210, 2021, 1112, \\
&\quad\,\,\, 1010, 2000\}=\mathcal{O}(2021)\cong\mathrm{M}^{(2,1,1)},\\
\mathcal{O}(1210)&=\{1021, 2221, 1111, 1000, 2101, 2200, 0211, 1102, 1210, 0001, \\
&\quad\,\,\, 0202, 1222\}=\mathcal{O}(1222)\cong\mathrm{M}^{(2,1,1)}.
\end{align*}
The number of orbits is  $3+6+3+3={4+3-1\choose 3-1}$.
\end{example}

\subsection{The Nichols algebra $\mathfrak{B}(W_{X, r})$} 
Let $(X, r)$ be a non-degenerate  involutive  solution of the Yang-Baxter equation, $|X|=m\in\Bbb Z^{\geq 2}$.
The braided vector space $W_{X,r}$ is defined in Definition \ref{BraidedVectorSpace}. 
In this section, 
we calculate the dimension and the Gelfand–Kirillov dimension of the Nichols algebra
$\mathfrak{B}(W_{X,r})$ under some given conditions.

\begin{lemma}
For any $j, k\in X$, $R_{D(j), j}=R_{D\tau_k(j), \tau_k(j)}$.
\end{lemma}
\begin{proof}
Let $i=D(j)$, then $\sigma_i(j)=i$ and $\tau_j(i)=j$. From 
the formula \eqref{YBEquation}, we have 
$R_{i,j}=R_{\tau_{\sigma_j(k)}(i), \tau_k(j)}$.  
According to the formula \eqref{FormulaeYBE1}, we obtain
\begin{align*}
\tau_{\sigma_j(k)}(i)
&=\left(\tau_{\tau_k(j)}\right)^{-1}\tau_k\tau_j(i)
=\left(\tau_{\tau_k(j)}\right)^{-1}\tau_k(j)=D\tau_k(j).
\end{align*}
\end{proof}

\begin{theorem}\label{MainTheorem}
Let $(X, r)$ be a non-degenerate  involutive  solution of the Yang-Baxter equation, $|X|=m\in\Bbb Z^{\geq 2}$.
If $q=R_{D(i),i}\in\Bbb{G}_n$ for all $i\in X$, $n\geq 2$,  and 
$R_{i,j}R_{\sigma_i(j),\tau_j(i)}=1$ for all $i$, $j\in X$ with 
$i\neq D(j)$, then 
\[
\dim \mathfrak{B}(W_{X, r})=n^m, 
\]
and the relations of the Nichols algebra $\mathfrak{B}(W_{X, r})$ are given by
\begin{align}
w_iw_j-R_{i,j} w_{\sigma_i(j)}w_{\tau_j(i)}&=0,\quad D(j)\neq i,\label{IdealRelation1}\\
w_{D^{n-1}(i)}w_{D^{n-2}(i)}\cdots w_{D(i)} w_{i}&=0,\quad \forall i, j\in X.
\label{IdealRelation2}
\end{align}
\end{theorem}
\begin{proof}
The relations \eqref{IdealRelation1} hold since $R_{i,j}R_{\sigma_i(j),\tau_j(i)}=1$ 
for all $i$, $j\in X$ with  $i\neq D(j)$. And the relations \eqref{IdealRelation2} 
hold because 
\[
\mathfrak{S}_n\left(w_{D^{n-1}(i)}w_{D^{n-2}(i)}\cdots w_{D(i)} w_{i}\right)
=(n)_q^!w_{D^{n-1}(i)}w_{D^{n-2}(i)}\cdots w_{D(i)} w_{i}=0.
\]

Let  $n^\prime \in \Bbb Z^+$, 
$\lambda=(\lambda_1, \lambda_2,\cdots, \lambda_k)\in\mathcal{P}(n^\prime,m)$, 
and  $x=\Psi_{\lambda_1}(a_1)\cdots \Psi_{\lambda_k}(a_k)$ be  a $\lambda$-element. 
We claim 
\begin{enumerate}
\item For any $y\in\mathcal O(x)$, $w_y\in\k^\times w_x$;
\item $w_x=0$ if and only if $\lambda_1\geq n$.
\end{enumerate}

For any 
$y\in\mathcal O(x)$, there exists an element 
$s=s_{i_t}s_{i_{t-1}}\cdots s_{i_2}s_{i_1}\in\Bbb S_{n^\prime}$ such that $s\cdot x=y$.
Denote 
\[
y_{k+1}=s_{i_{k+1}}\cdot y_k, \quad 0\leq k\leq t-1, \quad y_0=x,\quad y_t=y, 
\]
and $y_k=y_k[1]y_k[2]\cdots y_k[n^\prime]\in X^{n^\prime}$. 
According to \eqref{IdealRelation1},  we have 
\[
w_{y_k}-R_{y_k[i_{k+1}], y_k[i_{k+1}+1]}w_{y_{k+1}}=0,\quad 
0\leq k\leq t-1, 
\]
which implies that 
$
w_x=\prod_{k=0}^{t-1}R_{y_k[i_{k+1}], y_k[i_{k+1}+1]} w_y\in \k^\times w_y.
$

Since $\mathfrak{G}(x,x)=\Bbb S_{\lambda_1}\times \cdots \times \Bbb S_{\lambda_k}$, 
there exists a unique element $\theta_y\in\mathrm{shuffle}(\lambda_1,\cdots,\lambda_k)$
for any $y\in\mathcal{O}(x)$ such that $\theta_y\cdot x=y$.
This implies that there exists a parameter $\xi_y\in\k^\times$
such that $\mathcal{T}_{\theta_y}(w_x)=\xi_y w_y$. 
We have 
\[
\mathfrak{S}_{n^\prime}(w_x)
=(\lambda_1)_q^!(\lambda_2)_q^!\cdots (\lambda_k)_q^!
\sum_{y\in\mathcal O(x)} \xi_y w_y, 
\]
which implies that  $w_x=0$ if and only if  $\lambda_1\geq n$. 

According to  Lemma \ref{AllElement}, 
\[
T(W_{X,r})=\k\oplus \bigoplus_{n^\prime\in\Bbb Z^+, 
\lambda\in\mathcal{P}(n^\prime, m), y\in\mathcal{B}(\lambda)}\k w_y.  
\]
Now we see every orbit of $X^{n^\prime}$
contributes at most one dimension to the Nichols algebra and 
those that vanish correspond to partitions $\lambda=(\lambda_1,\cdots)$ with 
$\lambda_1\geq n$.

In case $r(i, j)=(j, i)$ for all $i, j\in X$,  $\dim \mathfrak{B}(W_{X, r})=n^m$ since it is of Cartan type 
$A_1\times \cdots \times A_1$($m$ copies).  Notice that the dimension of the Nichols algebra $\mathfrak{B}(W_{X, r})$ only relies on the parameters $m$ and $n$,  we have 
$\dim \mathfrak{B}(W_{X, r})=n^m$ for any non-degenerate involutive solution $(X, r)$. 

According to the above proof,  relations in part (1) are deduced from 
\eqref{IdealRelation1}. It is obvious that relations in part (2) are deduced from 
\eqref{IdealRelation2}. So there are no more new relations
in the Nichols algebra. 
\end{proof}

%\begin{corollary}
%The relations of the  Nichols algebra $\mathfrak{B}(W_{X, r})$ of dimension $n^m$ are given by 
%\begin{align}
%w_iw_j-R_{i,j} w_{\sigma_i(j)}w_{\tau_j(i)}&=0,\quad D(j)\neq i,\label{IdealRelation1}\\
%w_{D^{n-1}(i)}w_{D^{n-2}(i)}\cdots w_{D(i)} w_{i}&=0,\quad \forall i, j\in X.
%\label{IdealRelation2}
%\end{align}
%\end{corollary}
\begin{remark}
It is obvious  that $\mathfrak{B}(W_{X,r})$  is not of group type in general, 
see  Example \ref{NotGroupType}.
The theorem generalizes a result in \cite{Yuxing2020}, which associated  the Nichols algebras 
 of squared dimension with  Pascal's triangle. The Nichols algebras of squared dimension
 appeared first in Andruskiewitsch and Giraldi's work \cite{Andruskiewitsch2018}, and they 
 have realizations in the Yetter-Drinfeld categories
 of the Suzuki Hopf algebras $A_{Nn}^{\mu\lambda}$, see 
\cite{Shi2020even} \cite{Shi2020odd} \cite{Shi2019}.
\end{remark}
%\begin{proof}
%From the proof of Theorem \ref{MainTheorem},  it follows that relations 
%\eqref{IdealRelation1} and \eqref{IdealRelation2}
%are in the kernel of the symmetrizer. We prove there are no new relations in below. 
%It is enough to prove the following two classes of relations are deduced from 
%\eqref{IdealRelation1} and \eqref{IdealRelation2}. For any $n^\prime\in \Bbb Z^+$, 
%$x$ is a $\lambda$-element for $\lambda\in \mathcal{P}(n^\prime, m)$, we have 
%\begin{enumerate}
%\item For any $y\in\mathcal O(x)$, $w_y\in\k^\times w_x$;
%\item $w_x=0$, if $\lambda_1\geq n$, where $\lambda=(\lambda_1,\cdots, \lambda_m)$.
%\end{enumerate}
%It is obvious that relations in part (2) are from \eqref{IdealRelation2}. For any 
%$y\in\mathcal O(x)$, there exists an element 
%$s=s_{i_t}s_{i_{t-1}}\cdots s_{i_2}s_{i_1}\in\Bbb S_{n^\prime}$ such that $s\cdot x=y$.
%Denote 
%\[
%y_{k+1}=s_{i_{k+1}}\cdot y_k, \quad 0\leq k\leq t-1, \quad y_0=x,\quad y_t=y, 
%\]
%and $y_k=y_k[1]y_k[2]\cdots y_k[n^\prime]\in X^{n^\prime}$. 
%According to \eqref{IdealRelation1},  we have 
%\[
%w_{y_k}-R_{y_k[i_{k+1}], y_k[i_{k+1}+1]}w_{y_{k+1}}=0,\quad 
%0\leq k\leq t-1, 
%\]
%which implies that 
%\[
%w_x=\prod_{k=0}^{t-1}R_{y_k[i_{k+1}], y_k[i_{k+1}+1]} w_y\in \k^\times w_y.
%\]
%\end{proof}

\begin{corollary}
Let $q=R_{D(i),i}$ for all $i\in X$, $q\notin\Bbb{G}_n$ for all $n\geq 2$,  and 
\[
R_{i,j}R_{\sigma_i(j),\tau_j(i)}=1,\quad  \forall i, j\in X,\quad  
D(j)\neq i, 
\]
then 
the Gelfand–Kirillov dimension of 
$\mathfrak{B}(W_{X, r})$ is  $m$.
\end{corollary}
\begin{proof}
Since $w_y\in \k^{\times} w_x$ for any $y\in\mathcal{O}(x)\subseteq X^k$, 
$k\in\Bbb Z^+$, 
we have 
\begin{align*}
\mathrm{GKdim}  \mathfrak{B}(W_{X, r})
&=\lim_{n\to\infty}\log_n\sum_{k=0}^n{k+m-1\choose m-1}=m.
\end{align*}
\end{proof}
\begin{remark}
The result $\mathrm{GKdim}  \mathfrak{B}(W_{X,r})=2$ with $X=\Bbb Z_2$, 
$r(i,j)=(j-1, i+1)\in \Bbb Z_2\times \Bbb Z_2$, 
 was obtained first by 
Andruskiewitsch and Giraldi  \cite{Andruskiewitsch2018}.
Besides, the Nichols algebras of  Gelfand–Kirillov dimension $m$
were also studied by Gateva-Ivanova under different names 
\cite{zbMATH02123087}
\cite{zbMATH07405634}.
\end{remark}

%\begin{problem}
%For those Nichols algebras of dimension $n^m$, is it possible to obtain $\mathfrak{B}(W_{m,\tau})$
%with $\tau\neq \rm{id}$ from $\mathfrak{B}(W_{m,\rm{id}})$
%by some Hopf $2$-cocycle deformation or Drinfeld twist?
%\end{problem}
%\begin{remark}
%As for the case of infinite dimension, Masuoka provided an example, see 
%\cite{Shi2020even} \cite{MR1800713}. 
%A negative information is that the Suzuki Hopf algebras $A_{Nn}^{\mu\lambda}$ are 
%not categorically Morita-equivalent to group algebras in general, 
%for example  $A_{12}^{+-}$ \cite[Section 5.2]{MR2386730}.   
%\end{remark}

\begin{theorem}
Let $(X, r)$ be a non-degenerate  involutive  solution of the Yang-Baxter equation, $|X|=m\in\Bbb Z^{\geq 2}$.
Suppose $(X, r)$ is decomposable as 
\[
X=X_1\cup X_2\cup\cdots \cup X_t, 
\]
$q_k=R_{D(i_k),i_k}$ for all $i_k\in X_k$ and $1\leq k\leq t$.
If $q_k\in \Bbb G_{n_k}$ for $1\leq k\leq t$ and $n_k\geq 2$, 
$R_{i,j}R_{\sigma_i(j),\tau_j(i)}=1$ for all $i$, $j\in X$ with 
$i\neq D(j)$, then 
$$\dim\mathfrak{B}(W_{X,r})=n_1^{|X_1|}n_2^{|X_2|}\cdots n_t^{|X_t|}.$$
\end{theorem}
\begin{proof}
As for $1\leq k\leq t$, $M_k=\bigoplus_{a\in X_k}\k w_a$ is a braided vector subspace of $W_{X,r}$.
According to Theorem \ref{MainTheorem},  we have $\dim \mathfrak{B}(M_k)=n_k^{|X_k|}$. 
The proof is finished by Lemma \ref{TensorNicholsAlg}.
\end{proof}
\begin{remark}
If some elements of $\{q_k\}_{1\leq k\leq t}$ are not  roots of unity, then 
\[
0<\rm{GKdim}\, \mathfrak{B}(W_{X,r})\leq m.
\]
In fact, $\rm{GKdim}\, \mathfrak{B}(W_{X,r})$ can be calculated by  Lemma \ref{TensorNicholsAlg}. 
\end{remark}

\begin{example}\cite{Andruskiewitsch2018} \cite{Shi2020even} \cite{Shi2020odd}
\label{NotGroupType}
Let $(\Bbb Z_2, r)$ with 
 $r(i,j)=(j-1,i+1)\in\Bbb Z_2\times \Bbb Z_2$ be a permutation solution. 
$W_{\Bbb Z_2,r}=\bigoplus_{i\in\Bbb{Z}_2}\k w_i$ 
is a braided vector space defined as
\begin{align*}
c(w_ 0 \otimes w_ 0 )&= a w_ 1 \otimes w_ 1 ,&
c(w_ 0 \otimes w_ 1 )&=q w_ 0 \otimes w_ 1 ,\\
c(w_ 1 \otimes w_ 0 )&= q w_ 1 \otimes w_ 0 ,&
c(w_ 1 \otimes w_ 1 )&= e w_ 0 \otimes w_ 0 ,
\end{align*}
where $aqe\in\k^\times$.  If $ae=q^2$, then 
$\mathfrak{B}(W_{\Bbb Z_2,r})$ is of diagonal type and
\[
\dim \mathfrak{B}(W_{\Bbb Z_2,r})
=\left\{\begin{array}{ll}
4, & q=-1,\quad (\text{Cartan type $A_1\times A_1$}),\\
27, & q^3=1\neq q,\quad  (\text{Cartan type $A_2$}),\\
\infty, & \text{otherwise}.
\end{array}\right.
\]
Since  the braiding of $\mathfrak{B}(W_{\Bbb Z_2,r})$ is rank $2$, 
$\mathfrak{B}(W_{\Bbb Z_2,r})$ is of  group type if and only if   it is of diagonal type. 
If $q^2\neq ae$, $\mathfrak{B}(W_{\Bbb Z_2,r})$ is obviously not of group type,
\begin{align}\label{fomulaeVabe}
\dim\mathfrak{B}(W_{\Bbb Z_2,r})
=\left\{\begin{array}{ll}
4n, &q=-1, ae\in\Bbb{G}_n,\\
n^2, &ae=1, q\in\Bbb{G}_n\,\,\text{for}\,\,n\geq 2,\\
\infty, & q^2=(ae)^{-1},  q\in\Bbb{G}_{n}\,\, \text{for}\,\, n\geq 3,\\
\infty,   & q\notin \Bbb{G}_n \,\,\text{for}\,\,n\geq 2,\\
\text{unknown}, & otherwise.
\end{array}\right.
\end{align}
\end{example}

\begin{example}
Let $(\Bbb Z_3, r)$ with 
 $r(i,j)=(j-1,i+1)\in\Bbb Z_3\times \Bbb Z_3$ be a permutation solution. 
Then $W_{\Bbb Z_3, r}=\bigoplus_{i\in\Bbb{Z}_3}\k w_i$  
is a braided vector space defined as 
\begin{align*}
c(w_ 0 \otimes w_ 0 )&= a w_ 2 \otimes w_ 1 ,&
c(w_ 0 \otimes w_ 1 )&=q w_ 0 \otimes w_ 1 ,&
c(w_ 0 \otimes w_ 2 )&= d w_ 1 \otimes w_ 1 ,\\
c(w_ 1 \otimes w_ 0 )&= e w_ 2 \otimes w_ 2 ,&
c(w_ 1 \otimes w_ 1 )&= f w_ 0 \otimes w_ 2 ,&
c(w_ 1 \otimes w_ 2 )&=q w_ 1 \otimes w_ 2 ,\\
c(w_ 2 \otimes w_ 0 )&=q w_ 2 \otimes w_ 0 ,&
c(w_ 2 \otimes w_ 1 )&= \frac{df}{a} w_ 0 \otimes w_ 0 ,&
c(w_ 2 \otimes w_ 2 )&=\frac{df}{e} w_ 1 \otimes w_ 0, 
\end{align*}
where $adefq\in\k^\times$. 
If $df=1$, $q\in\Bbb{G}_n$ for $n\geq 2$, then $\dim\mathfrak{B}(W_{\Bbb Z_3, r})=n^3$ and 
\begin{align*}
w_0^2-aw_2w_1=0,\quad w_0w_2-d w_1^2=0,\quad  w_1w_0-ew_2^2=0, \\
w_iw_{i+1}\cdots w_{i+n-1}=0,\quad \forall i\in\Bbb{Z}_3.
\end{align*}
%In this case, $\dim\mathfrak{B}(W_{3,\tau})=n^3$. 
\end{example}

\begin{example}
Let $(\Bbb Z_4, r)$ with 
 $r(i,j)=(j-1,i+1)\in\Bbb Z_4\times \Bbb Z_4$ be a permutation solution. 
Then $W_{\Bbb Z_4, r}=\bigoplus_{i\in\Bbb{Z}_4}\k w_i$ 
is a braided vector space defined as 
\begin{align*}
c(w_0 \otimes w_0 )&= x_1 w_3 \otimes w_1 ,&
c(w_0 \otimes w_1 )&= q w_0 \otimes w_1 ,\\
c(w_0 \otimes w_2 )&= x_2 w_1 \otimes w_1 ,&
c(w_0 \otimes w_3 )&= x_3 w_2 \otimes w_1 ,\\
c(w_1 \otimes w_0 )&= x_4 w_3 \otimes w_2 ,&
c(w_1 \otimes w_1 )&= x_5 w_0 \otimes w_2 ,\\
c(w_1 \otimes w_2 )&= q w_1 \otimes w_2 ,&
c(w_1 \otimes w_3 )&= x_6 w_2 \otimes w_2 ,\\
c(w_2 \otimes w_0 )&= \frac{x_2x_4}{x_3} w_3 \otimes w_3 ,&
c(w_2 \otimes w_1 )&= \frac{x_2x_4x_5}{x_1x_6} w_0 \otimes w_3 ,\\
c(w_2 \otimes w_2 )&= \frac{x_2x_5}{x_6} w_1 \otimes w_3 ,&
c(w_2 \otimes w_3 )&= q w_2 \otimes w_3 ,\\
c(w_3 \otimes w_0 )&= q w_3 \otimes w_0 ,&
c(w_3 \otimes w_1 )&= \frac{x_2x_5}{x_1} w_0 \otimes w_0 ,\\
c(w_3 \otimes w_2 )&= \frac{x_2x_3x_5}{x_1x_6} w_1 \otimes w_0 ,&
c(w_3 \otimes w_3 )&= \frac{x_3x_5}{x_4} w_2 \otimes w_0,
\end{align*}
where $qx_1x_2x_3x_4x_5x_6\in \k^\times$. 
If $x_2x_5=1$, $x_1x_6=x_2x_3x_4x_5$,  $q\in\Bbb{G}_n$ for $n\geq 2$, 
then 
\begin{align*}
w_0^2-x_1w_3w_1=0,\quad  w_0w_2-x_2w_1^2=0,\quad w_0w_3-x_3w_2w_1=0,\\
w_1w_0-x_4 w_3w_2=0\quad w_1w_3-x_6w_2^2=0,\quad 
w_2w_0-\frac{x_2x_4}{x_3}w_3^2=0,\\
w_iw_{i+1}\cdots w_{i+n-1}=0,\quad \forall i\in\Bbb{Z}_4.
\end{align*}
In this case, $\dim \mathfrak{B}(W_{4,\tau})=n^4$.
\end{example}

\begin{example}
Let $(\Bbb Z_4, r)$ with 
 $r(i,j)=(j-2,i+2)\in\Bbb Z_4\times \Bbb Z_4$ be a permutation solution. 
Then $W_{\Bbb Z_4, r}=\bigoplus_{i\in\Bbb{Z}_4}\k w_i$  
is a braided vector space defined as
\begin{align*}
c(w_0 \otimes w_0 )&= x_1 w_2 \otimes w_2 ,&
c(w_0 \otimes w_1 )&= x_2 w_3 \otimes w_2 ,\\
c(w_0 \otimes w_2 )&= q_1 w_0 \otimes w_2 ,&
c(w_0 \otimes w_3 )&= x_3 w_1 \otimes w_2 ,\\
c(w_1 \otimes w_0 )&= x_4 w_2 \otimes w_3 ,&
c(w_1 \otimes w_1 )&= x_5 w_3 \otimes w_3 ,\\
c(w_1 \otimes w_2 )&= x_6 w_0 \otimes w_3 ,&
c(w_1 \otimes w_3 )&= q_2 w_1 \otimes w_3 ,\\
c(w_2 \otimes w_0 )&= q_1 w_2 \otimes w_0 ,&
c(w_2 \otimes w_1 )&= x_7 w_3 \otimes w_0 ,\\
c(w_2 \otimes w_2 )&= \frac{x_1x_7x_9}{x_2x_3} w_0 \otimes w_0 ,&
c(w_2 \otimes w_3 )&= x_9 w_1 \otimes w_0 ,\\
c(w_3 \otimes w_0 )&= \frac{x_3x_6}{x_7} w_2 \otimes w_1 ,&
c(w_3 \otimes w_1 )&= q_2 w_3 \otimes w_1 ,\\
c(w_3 \otimes w_2 )&= \frac{x_4x_9}{x_2} w_0 \otimes w_1 ,&
c(w_3 \otimes w_3 )&= \frac{x_3x_5x_9}{x_2x_7} w_1 \otimes w_1 ,
\end{align*}
where $q_1q_2x_1x_2x_3x_4x_5x_6x_7x_9\in \k^\times$.
If $x_1^2x_7x_9=x_2x_3$, $x_4x_9=1$, $x_3x_6=1$, $x_2x_7=x_3x_5^2x_9$, 
$q_i\in\Bbb{G}_{n_i}$ for $n_i\geq 2$, $i\in\{1,2\}$, then $\dim \mathfrak{B}(W_{4,\tau})=n_1^2n_2^2$ and 
\begin{align*}
w_0^2-x_1w_2^2=0,\quad  w_0w_1-x_2w_3w_2=0,\quad
w_0w_3-x_3w_1w_2=0,\\
w_1w_0-x_4w_2w_3=0,\quad w_2w_1-x_7w_3w_0=0,\quad
w_1^2-x_5w_3^2=0,\\
w_iw_{i+2}\cdots w_{i+2(n_1-1)}=0,\quad 
w_jw_{j+2}\cdots w_{j+2(n_2-1)}=0,\quad i\in\{0, 2\},\quad  j\in\{1, 3\}.
\end{align*}
%In this case, $\dim \mathfrak{B}(W_{4,2})=n^4$.
\end{example}

%\begin{example}
%Let $qx_1x_2x_3x_4x_5x_8\in\k^\times $, then 
%$W_{4,\tau}=\bigoplus_{i\in\Bbb{Z}_4}\k w_i$ with $\tau=(0\,2\,1\,3)$ 
%is a braided vector space defined as
%\begin{align*}
%c(w_ 0 \otimes w_ 0 )&= x_1 w_ 3 \otimes w_ 2 ,&
%c(w_ 0 \otimes w_ 1 )&= x_2 w_ 2 \otimes w_ 2 ,\\
%c(w_ 0 \otimes w_ 2 )&= q w_ 0 \otimes w_ 2 ,&
%c(w_ 0 \otimes w_ 3 )&= x_4 w_ 1 \otimes w_ 2 ,\\
%c(w_ 1 \otimes w_ 0 )&= x_5 w_ 3 \otimes w_ 3 ,&
%c(w_ 1 \otimes w_ 1 )&= x_1x_3x_4^{-1} w_ 2 \otimes w_ 3 ,\\
%c(w_ 1 \otimes w_ 2 )&= x_3x_5x_2^{-1} w_ 0 \otimes w_ 3 ,&
%c(w_ 1 \otimes w_ 3 )&= q w_ 1 \otimes w_ 3 ,\\
%c(w_ 2 \otimes w_ 0 )&= x_4x_5x_2^{-1} w_ 3 \otimes w_ 1 ,&
%c(w_ 2 \otimes w_ 1 )&= q w_ 2 \otimes w_ 1 ,\\
%c(w_ 2 \otimes w_ 2 )&= x_1x_8x_2^{-1} w_ 0 \otimes w_ 1 ,&
%c(w_ 2 \otimes w_ 3 )&= x_4x_8x_3^{-1} w_ 1 \otimes w_ 1 ,\\
%c(w_ 3 \otimes w_ 0 )&= q w_ 3 \otimes w_ 0 ,&
%c(w_ 3 \otimes w_ 1 )&= x_3 w_ 2 \otimes w_ 0 ,\\
%c(w_ 3 \otimes w_ 2 )&= x_8 w_ 0 \otimes w_ 0 ,&
%c(w_ 3 \otimes w_ 3 )&= x_1x_8x_5^{-1} w_ 1 \otimes w_ 0.
%\end{align*}
%If $x_1x_8=1$, $x_3x_4x_5=x_2$, $q\in\Bbb G_n$ for $n\geq 2$, then 
%$\dim \mathfrak{B}(W_{4,\tau})=n^4$ and 
%\begin{align*}
%w_0^2-x_1 w_3 w_2=0,\quad 
%w_0w_1-x_2 w_2^2=0,\quad 
%w_0w_3-x_4 w_1w_2=0,\\
%w_1w_0-x_5 w_3^2=0,\quad 
%w_1^2-x_1x_3x_4^{-1} w_2w_3=0,\quad 
%w_3w_1-x_3 w_2w_0=0,\\
%w_iw_{\tau(i)}\cdots w_{\tau^{n-1}(i)}=0,\quad \forall i \in\Bbb{Z}_4.
%\end{align*}
%\end{example}

\begin{example}
Let $X=\{1, 2, 3, 4\}$ and $r(i, j)=(\sigma_i(j), \tau_j(i))$, where 
\begin{align*}
\sigma_1&=(3\,4),&\sigma_2&=(1\,3\,2\,4),
& \sigma_3&=(1\,4\,2\,3), &\sigma_4&=(1\,2),\\
\tau_1&=(2\,4),& \tau_2&=(1\,4\,3\,2),& \tau_3&=(1\,2\,3\,4),& \tau_4&=(1\,3).
\end{align*}
Let $(x_2x_4x_6)^2=(x_5^2x_8)^2$, $qx_2x_3x_4x_5x_6x_8\in\k^\times$, then 
$W_{X, r}=\bigoplus_{i\in X}\k w_i$  
is a braided vector space defined as 
\begin{align*}
c(w_ 1 \otimes w_ 1 )&= q w_ 1 \otimes w_ 1 ,&
c(w_ 1 \otimes w_ 2 )&= x_2 w_ 2 \otimes w_ 4 ,\\
c(w_ 1 \otimes w_ 3 )&= x_3 w_ 4 \otimes w_ 2 ,&
c(w_ 1 \otimes w_ 4 )&= x_4 w_ 3 \otimes w_ 3 ,\\
c(w_ 2 \otimes w_ 1 )&= x_5 w_ 3 \otimes w_ 4 ,&
c(w_ 2 \otimes w_ 2 )&= x_6 w_ 4 \otimes w_ 1 ,\\
c(w_ 2 \otimes w_ 3 )&= q w_ 2 \otimes w_ 3 ,&
c(w_ 2 \otimes w_ 4 )&= x_8 w_ 1 \otimes w_ 2 ,\\
c(w_ 3 \otimes w_ 1 )&= x_5^2x_8(x_4x_6)^{-1} w_ 4 \otimes w_ 3 ,&
c(w_ 3 \otimes w_ 2 )&= q w_ 3 \otimes w_ 2 ,\\
c(w_ 3 \otimes w_ 3 )&= x_3x_5x_8(x_2x_4)^{-1} w_ 1 \otimes w_ 4 ,&
c(w_ 3 \otimes w_ 4 )&= x_3x_8x_2^{-1} w_ 2 \otimes w_ 1 ,\\
c(w_ 4 \otimes w_ 1 )&= x_3x_5x_8(x_2x_6)^{-1} w_ 2 \otimes w_ 2 ,&
c(w_ 4 \otimes w_ 2 )&= x_5x_8x_2^{-1} w_ 1 \otimes w_ 3 ,\\
c(w_ 4 \otimes w_ 3 )&= x_2x_4x_6x_5^{-2} w_ 3 \otimes w_ 1 ,&
c(w_ 4 \otimes w_ 4 )&= q w_ 4 \otimes w_ 4.
\end{align*}
If $x_2x_8=1$,  $x_2=x_3x_5x_8$, $q\in\Bbb G_n$ for $n\geq 2$, then 
$\dim \mathfrak{B}(W_{X,r})=n^4$ and 
\begin{align*}
w_1w_2-x_2 w_2w_4=0,\quad w_1w_3-x_3 w_4w_2=0,\quad 
w_1w_4-x_4 w_3^2=0,\\
w_2w_1-x_5 w_3w_4=0,\quad 
w_2^2-x_6 w_4w_1=0,\quad
w_3w_1-\frac{x_5^2}{x_2x_4x_6}w_4w_3=0,\\
w_{D^{n-1}(i)}w_{D^{n-2}(i)}\cdots w_{D(i)}w_i=0,\quad \forall i\in X, D=(2\, 3)\in\Bbb S_4.
\end{align*}
\end{example}

\section{Further research}
\begin{question}
How to classify all finite dimensional Nichols algebras associated to non-degenerate 
involutive solutions of the Yang-Baxter equation?
\end{question}

We say two Nichols algebras 
$\mathfrak{B}(V_1)$  and $\mathfrak{B}(V_2)$ are
Morita equivalent if   there exist two Hopf algebras $H_1$ and $H_2$ such that 
$\mathcal{F}: {}_{H_1}^{H_1}\mathcal{YD}\rightarrow {}_{H_2}^{H_2}\mathcal{YD}$
is isomorphic as braided tensor categories and $\mathcal{F}(V_1)=V_2$.
\begin{question}
In the sense of Morita equivalence, would it be possible to classify those finite dimensional Nichols algebras obtained in the 
paper? 
In particular, which of them are Morita equivalent to Nichols algebras of group type?
\end{question}
\begin{remark}
 The Nichols algebra $\mathfrak{B}(W_{\Bbb Z_2, r})$ of squared dimension 
 can be realized in the Yetter-Drinfeld categories
 of the Suzuki Hopf algebras $A_{Nn}^{\mu\lambda}$, 
 see Example \ref{NotGroupType}.
Masuoka proved that  $A_{1n}^{++}$ is isomorphic to a $2$-cocycle deformation of $\k^{D_{4n}}$ \cite{MR1800713}.
A negative information is that the Suzuki Hopf algebras $A_{Nn}^{\mu\lambda}$ are 
not categorically Morita equivalent to group algebras in general, 
for example  $A_{12}^{+-}$ \cite[Section 5.2]{MR2386730}.   
\end{remark}

\begin{question}
Realize those  finite dimensional Nichols algebras obtained in the 
paper in categories of Yetter-Drinfeld modules and use them to classify finite dimensional 
Hopf algebras according to the lifting method \cite{andruskiewitsch2001pointed}.
\end{question}

Let $(X, r)$ be a non-degenerate solution of the Yang-Baxter equation and $H_{X,r}$ be
the group generated by $r$. Since $X$ is a finite non-empty set, we have 
$r^n=\rm{id}$ for some $n\in\Bbb Z^+$. The group $H_{X,r}$ acts on 
$X\times X$ and the orbits of this action are 
\[
\mathcal{O}(i,j)=\left\{r^k(i,j)\,\big|\, k\in \Bbb Z, i, j\in X\right\}.
\]
The set $X\times X$ is the disjoint union of the orbits under action of $H_{X,r}$. 
Let 
\[
l_n=\#\{\mathcal{O}(i,j):\,\, \mathcal{O}(i,j)\,\,\text{has $n$ elements}\}.
\]
Then $l_1+2l_2+3l_3+\cdots=|X|^2$. 
In case that $X$ is a rack, orbits and sizes of orbits of $X\times X$ were used to study 
Nichols algebras of group type with many quadratic relations \cite{grana2011nichols}. 
Denote 
\[
\Phi_{X,r}=(l_1,l_2,l_3,\cdots).
\]
\begin{conjecture}
Let $(X, r)$ be a non-degenerate indecomposable  solution of the Yang-Baxter equation
and 
\[
\dim \mathfrak{B}(W_{X,r})=m\quad (\text{or}\,\, \rm{GKdim}\,\mathfrak{B}(W_{X,r})=m).
\]
Suppose $(Y, r^\prime)$ is any non-degenerate solution of the Yang-Baxter equation 
with $$\Phi_{X,r}=\Phi_{Y,r^\prime},$$ 
then $\dim \mathfrak{B}(W_{Y,r^\prime})=m$  
(or $\rm{GKdim}\,\mathfrak{B}(W_{Y,r^\prime})=m$)
under some given conditions. 
\end{conjecture}
\begin{remark}
If $(X, r)$ is a non-degenerate indecomposable involutive  solution with 
$|X|=m$, 
then the conjecture holds for $\mathfrak{B}(W_{X,r})$ with dimension $n^m$ or 
the Gelfand–Kirillov dimension
$m$ by results of the paper. 
According to the conjecture, we find $8$ classes of $72$-dimensional Nichols algebras
which are presented in below. 
\end{remark}

Due to Akg\"{u}n, Mereb and Vendramin's enumeration 
of set-theoretical solutions to the Yang-Baxter equation \cite{zbMATH07506857},  many finite dimensional Nichols algebras are going to be obtained. 
Examples show that  there are analogue Nichols algebras of dimension $12$, $72$, $5184$, $1280$, $576$, 
$326592$ and $8294400$, which are listed in \cite[Table 9.1]{Heckenberger2015}. 
We are interested in the following question. 
\begin{question}
Would it be possible to describe those Nichols algebras listed in \cite[Table 9.1]{Heckenberger2015}  with some combinatoric approach and make a generalization?
For example, we obtain the following $8$ classes of $72$-dimensional Nichols algebras according to  the enumeration %of set-theoretical solutions to the Yang-Baxter equation  
in \cite{zbMATH07506857}, under the assistance of the software GAP. 
Is it possible to  describe  the $8$ classes of $72$-dimensional Nichols algebras 
through some combinatoric approach in a unified way?
\end{question}
\begin{example}\cite{Grana2000}
 Let $qx_2x_3x_7x_8\in\k^\times$ and $(x_3x_8)^2=q^4$. 
 $W_1=\bigoplus_{i=1}^4\k w_i$ is a braided vector space,  with the braiding given by
 \begin{align*}
 c(w_ 1 \otimes w_ 1 )&= q w_ 1 \otimes w_ 1 ,&
c(w_ 1 \otimes w_ 2 )&= x_2 w_ 3 \otimes w_ 1 ,\\
c(w_ 1 \otimes w_ 3 )&= x_3 w_ 4 \otimes w_ 1 ,&
c(w_ 1 \otimes w_ 4 )&= q^3(x_2x_3)^{-1} w_ 2 \otimes w_ 1 ,\\
c(w_ 2 \otimes w_ 1 )&= q^3(x_7x_8)^{-1} w_ 4 \otimes w_ 2 ,&
c(w_ 2 \otimes w_ 2 )&= q w_ 2 \otimes w_ 2 ,\\
c(w_ 2 \otimes w_ 3 )&= x_7 w_ 1 \otimes w_ 2 ,&
c(w_ 2 \otimes w_ 4 )&= x_8 w_ 3 \otimes w_ 2 ,\\
c(w_ 3 \otimes w_ 1 )&= q^5(x_2x_3x_7x_8)^{-1} w_ 2 \otimes w_ 3 ,&
c(w_ 3 \otimes w_ 2 )&= qx_2x_8^{-1} w_ 4 \otimes w_ 3 ,\\
c(w_ 3 \otimes w_ 3 )&= q w_ 3 \otimes w_ 3 ,&
c(w_ 3 \otimes w_ 4 )&= qx_7x_3^{-1} w_ 1 \otimes w_ 3 ,\\
c(w_ 4 \otimes w_ 1 )&= q^4(x_3x_7x_8)^{-1} w_ 3 \otimes w_ 4 ,&
c(w_ 4 \otimes w_ 2 )&= x_2x_7q^{-1} w_ 1 \otimes w_ 4 ,\\
c(w_ 4 \otimes w_ 3 )&= q^4(x_2x_3x_8)^{-1} w_ 2 \otimes w_ 4 ,&
c(w_ 4 \otimes w_ 4 )&= q w_ 4 \otimes w_ 4.
\end{align*}
If $q=-1$, $x_3x_8=1$, then $\dim \mathfrak(W_1)=72$ and  
\begin{align*}
w_1^2=w_2^2=w_3^2=w_4^2=0,\quad
w_1w_4+(x_2x_3)^{-1}w_2w_1+(x_2x_7)^{-1}w_4w_2=0,\\
w_1w_3-x_3 w_4w_1+x_3x_7^{-1} w_3w_4=0,\quad
w_1w_2-x_2 w_3w_1-x_7^{-1}w_2w_3=0,\\
w_2w_4-x_8 w_3w_2-x_2 w_4w_3=0,\quad
(w_3w_2w_1)^2+(w_2w_1w_3)^2+(w_1w_3w_2)^2=0.
\end{align*}
 \end{example}

\begin{example}
  Let $x_9^2=q^2=x_5^2$, $(x_4x_7)^2=q^4$, $qx_2x_3x_4x_5x_7x_9\in\k^\times$, 
  $W_2=\bigoplus_{i=1}^4\k w_i$ is a braided vector space with the braiding given by 
  \begin{align*}
  c(w_ 1 \otimes w_ 1 )&= q w_ 1 \otimes w_ 1 ,&
c(w_ 1 \otimes w_ 2 )&= x_2 w_ 3 \otimes w_ 4 ,\\
c(w_ 1 \otimes w_ 3 )&= x_3 w_ 4 \otimes w_ 2 ,&
c(w_ 1 \otimes w_ 4 )&= x_4 w_ 2 \otimes w_ 3 ,\\
c(w_ 2 \otimes w_ 1 )&= x_5 w_ 1 \otimes w_ 2 ,&
c(w_ 2 \otimes w_ 2 )&= x_2x_4x_7(x_3x_5)^{-1} w_ 3 \otimes w_ 3 ,\\
c(w_ 2 \otimes w_ 3 )&= x_7 w_ 4 \otimes w_ 1 ,&
c(w_ 2 \otimes w_ 4 )&= q w_ 2 \otimes w_ 4 ,\\
c(w_ 3 \otimes w_ 1 )&= x_9 w_ 1 \otimes w_ 3 ,&
c(w_ 3 \otimes w_ 2 )&= q w_ 3 \otimes w_ 2 ,\\
c(w_ 3 \otimes w_ 3 )&= x_3x_7q^{-1} w_ 4 \otimes w_ 4 ,&
c(w_ 3 \otimes w_ 4 )&= qx_4x_7(x_2x_9)^{-1} w_ 2 \otimes w_ 1 ,\\
c(w_ 4 \otimes w_ 1 )&= x_5x_9q^{-1} w_ 1 \otimes w_ 4 ,&
c(w_ 4 \otimes w_ 2 )&= qx_4x_7(x_3x_5)^{-1} w_ 3 \otimes w_ 1 ,\\
c(w_ 4 \otimes w_ 3 )&= q w_ 4 \otimes w_ 3 ,&
c(w_ 4 \otimes w_ 4 )&= x_4^2x_7(x_2x_9)^{-1} w_ 2 \otimes w_ 2.
  \end{align*}
  If $q=-1$, $x_4x_5x_7x_9=1$, 
   then $\dim \mathfrak{B}(W_2)=72$ and 
  \begin{align*}
  w_1^2=w_2w_4=w_3w_2=w_4w_3=0,\quad
  w_1w_2-x_2 w_3w_4-x_5^{-1} w_2w_1=0,\\
  w_1w_3-x_3 w_4w_2-x_9^{-1} w_3w_1=0,\quad
  w_1w_4-x_4 w_2w_3+x_4x_7 w_4w_1=0,\\
  w_3^2+x_3x_7 w_4^2-x_3(x_2x_9)^{-1} w_2^2=0,\\
  x_4 w_3w_1w_2w_1w_2w_3+\frac{x_3x_4}{x_2} \left[(w_2w_2w_1)^2+ (w_1w_2w_2)^2\right]+w_1w_3w_1w_2w_1w_4=0.
  \end{align*}
\end{example}

\begin{example}
 Let $(x_{3} x_{4} x_{7})^2=q^6$, $qx_{3} x_{4} x_6x_{7}\in\k^\times$. 
 $W_3=\bigoplus_{i=1}^4\k w_i$ is a braided vector space,  with the braiding given by
 \begin{align*}
c(w_ 1 \otimes w_ 1 )&= q w_ 1 \otimes w_ 1 ,&
c(w_ 1 \otimes w_ 2 )&= q^3(x_3x_4)^{-1} w_ 4 \otimes w_ 1 ,\\
c(w_ 1 \otimes w_ 3 )&= x_3 w_ 2 \otimes w_ 1 ,&
c(w_ 1 \otimes w_ 4 )&= x_4 w_ 3 \otimes w_ 1 ,\\
c(w_ 2 \otimes w_ 1 )&= x_7^3(qx_6)^{-1} w_ 4 \otimes w_ 4 ,&
c(w_ 2 \otimes w_ 2 )&= x_6 w_ 1 \otimes w_ 4 ,\\
c(w_ 2 \otimes w_ 3 )&= x_7 w_ 3 \otimes w_ 4 ,&
c(w_ 2 \otimes w_ 4 )&= q w_ 2 \otimes w_ 4 ,\\
c(w_ 3 \otimes w_ 1 )&= x_3x_7x_6^{-1} w_ 2 \otimes w_ 2 ,&
c(w_ 3 \otimes w_ 2 )&= q w_ 3 \otimes w_ 2 ,\\
c(w_ 3 \otimes w_ 3 )&= q^2x_3x_6(x_4x_7^2)^{-1} w_ 1 \otimes w_ 2 ,&
c(w_ 3 \otimes w_ 4 )&= x_3^2x_4x_7q^{-3} w_ 4 \otimes w_ 2 ,\\
c(w_ 4 \otimes w_ 1 )&= qx_4x_7(x_3x_6)^{-1} w_ 3 \otimes w_ 3 ,&
c(w_ 4 \otimes w_ 2 )&= x_3x_4^2x_7q^{-3} w_ 2 \otimes w_ 3 ,\\
c(w_ 4 \otimes w_ 3 )&= q w_ 4 \otimes w_ 3 ,&
c(w_ 4 \otimes w_ 4 )&= qx_4x_6x_7^{-2} w_ 1 \otimes w_ 3.
\end{align*}
If $q=-1=x_3x_4x_7$, then $\dim \mathfrak{B}(W_3)=72$ and 
\begin{align*}
w_1^2=w_2w_4=w_3w_2=w_4w_3=0,\quad
w_4w_1-x_7^{-1}w_1w_2-x_3^{-2}x_6^{-1}w_3^2=0,\\
w_1w_3-x_3w_2w_1-x_3x_7^3x_6^{-1}w_4^2=0,\quad
w_1w_4-x_4 w_3w_1-x_6^{-1} w_2^2=0,\\
w_2w_3-x_7 w_3w_4+x_3x_7 w_4w_2=0,\\
x_3^6 w_2^6+ w_3^6-x_3^6x_4x_6 w_1w_2^4w_3+x_3^9x_4^3x_6^3 (w_2w_1)^3=0.
\end{align*}
\end{example}

\begin{example}
 Let $qx_3x_4x_5x_6\in\k^\times$, $q^{8} x_{5}^{4}= x_{3}^{4} x_{4}^2 x_{6}^{6}$. 
 $W_4=\bigoplus_{i=1}^4\k w_i$ is a braided vector space,  with the braiding given by
  \begin{align*}
c(w_ 1 \otimes w_ 1 )&= q w_ 1 \otimes w_ 1 ,&
c(w_ 1 \otimes w_ 2 )&= x_3^2x_4x_5x_6^2q^{-5} w_ 4 \otimes w_ 4 ,\\
c(w_ 1 \otimes w_ 3 )&= x_3 w_ 2 \otimes w_ 2 ,&
c(w_ 1 \otimes w_ 4 )&= x_4 w_ 3 \otimes w_ 3 ,\\
c(w_ 2 \otimes w_ 1 )&= x_5 w_ 1 \otimes w_ 4 ,&
c(w_ 2 \otimes w_ 2 )&= x_6 w_ 4 \otimes w_ 1 ,\\
c(w_ 2 \otimes w_ 3 )&= q w_ 2 \otimes w_ 3 ,&
c(w_ 2 \otimes w_ 4 )&= q^3(x_3x_6)^{-1} w_ 3 \otimes w_ 2 ,\\
c(w_ 3 \otimes w_ 1 )&= q^4x_5(x_3x_4x_6^2)^{-1} w_ 1 \otimes w_ 2 ,&
c(w_ 3 \otimes w_ 2 )&= x_3^2x_4x_6^3(q^4x_5)^{-1} w_ 4 \otimes w_ 3 ,\\
c(w_ 3 \otimes w_ 3 )&= q^7x_5(x_3^2x_4^2x_6^3)^{-1} w_ 2 \otimes w_ 1 ,&
c(w_ 3 \otimes w_ 4 )&= q w_ 3 \otimes w_ 4 ,\\
c(w_ 4 \otimes w_ 1 )&= q^7x_5^2(x_3^3x_4x_6^4)^{-1} w_ 1 \otimes w_ 3 ,&
c(w_ 4 \otimes w_ 2 )&= q w_ 4 \otimes w_ 2 ,\\
c(w_ 4 \otimes w_ 3 )&= x_3x_6x_5^{-1} w_ 2 \otimes w_ 4 ,&
c(w_ 4 \otimes w_ 4 )&= q^8(x_3^3x_4x_6^3)^{-1} w_ 3 \otimes w_ 1.
 \end{align*}
 If $q=-1$, $x_3^2x_4x_6^3 = x_5^2$, then $\dim\mathfrak{B}(W_4)=72$ and
 \begin{align*}
 w_1^2=w_2w_3=w_3w_4=w_4w_2=0,\quad
 w_1w_2+x_5^3x_6^{-1} w_4^2-x_3x_4x_6^2x_5^{-1} w_3w_1=0,\\
 w_1w_3-x_3 w_2^2+x_3x_6 w_4w_1=0,\quad
 w_2w_1-x_5 w_1w_4+x_4x_5 w_3^2=0,\\
 w_3w_2+x_3x_6 w_2w_4-x_5 w_4w_3=0,\\
 w_2^6+x_3^{-3}(w_3w_1)^3-x_6^3x_5^{-3}(w_1w_2)^3+x_6 w_1w_2^4w_4-x_6^3x_5^{-3}(w_2w_1)^3=0.
 \end{align*}
 \end{example}
 
 \begin{example}
Let $x_4^2=q^2=x_3^2$, $(x_5x_6)^2=q^4$, $qx_2x_3x_4x_5x_6x_8\in\k^\times$.  
$W_5=\bigoplus_{i=1}^4\k w_i$ is a braided vector space,  with the braiding given by
\begin{align*}
c(w_ 1 \otimes w_ 1 )&= q w_ 1 \otimes w_ 1 ,&
c(w_ 1 \otimes w_ 2 )&= x_3x_4q^{-1} w_ 2 \otimes w_ 1 ,\\
c(w_ 1 \otimes w_ 3 )&= x_3 w_ 3 \otimes w_ 1 ,&
c(w_ 1 \otimes w_ 4 )&= x_4 w_ 4 \otimes w_ 1 ,\\
c(w_ 2 \otimes w_ 1 )&= x_5 w_ 4 \otimes w_ 3 ,&
c(w_ 2 \otimes w_ 2 )&= x_5x_8q^{-1} w_ 3 \otimes w_ 3 ,\\
c(w_ 2 \otimes w_ 3 )&= q w_ 2 \otimes w_ 3 ,&
c(w_ 2 \otimes w_ 4 )&= x_8 w_ 1 \otimes w_ 3 ,\\
c(w_ 3 \otimes w_ 1 )&= qx_5x_6(x_4x_8)^{-1} w_ 2 \otimes w_ 4 ,&
c(w_ 3 \otimes w_ 2 )&= q^2x_3(x_2x_5)^{-1} w_ 1 \otimes w_ 4 ,\\
c(w_ 3 \otimes w_ 3 )&= q^2x_3x_6(x_2x_4x_8)^{-1} w_ 4 \otimes w_ 4 ,&
c(w_ 3 \otimes w_ 4 )&= q w_ 3 \otimes w_ 4 ,\\
c(w_ 4 \otimes w_ 1 )&= qx_2x_6^{-1} w_ 3 \otimes w_ 2 ,&
c(w_ 4 \otimes w_ 2 )&= q w_ 4 \otimes w_ 2 ,\\
c(w_ 4 \otimes w_ 3 )&= x_6 w_ 1 \otimes w_ 2 ,&
c(w_ 4 \otimes w_ 4 )&= x_2 w_ 2 \otimes w_ 2.
\end{align*}
If $q=-1$, $x_3x_4x_5x_6=1$, then $\dim\mathfrak{B}(W_5)=72$ and 
\begin{align*}
w_1^2=w_2w_3=w_3w_4=w_4w_2=0,\quad
w_2w_1+x_5x_6 w_1w_2-x_5 w_4w_3=0,\\
w_2w_4-x_8 w_1w_3 +x_3x_8 w_3w_1=0,\quad
w_1w_4-x_4 w_4w_1-x_2x_5x_3^{-1} w_3w_2=0,\\
w_4^2-x_2 w_2^2-x_2x_5x_8 w_3^2=0,\\
(w_1w_2w_2)^2-x_5x_8x_3^{-1}w_1w_3w_2^2w_1w_3+(w_2w_2w_1)^2+x_5 w_3w_2^2w_1w_2w_4=0.
\end{align*}
\end{example}

\begin{example}\cite{Grana2000}
Let $qx_1x_2x_3x_5\in\k^\times$, $(x_3x_5)^2=q^4$. 
$W_6=\bigoplus_{i=1}^4\k w_i$ is a braided vector space,  with the braiding given by
\begin{align*}
c(w_ 1 \otimes w_ 1 )&= q w_ 1 \otimes w_ 1 ,&
c(w_ 1 \otimes w_ 2 )&= q^3(x_2x_3)^{-1} w_ 2 \otimes w_ 4 ,\\
c(w_ 1 \otimes w_ 3 )&= x_1 w_ 3 \otimes w_ 2 ,&
c(w_ 1 \otimes w_ 4 )&= x_5 w_ 4 \otimes w_ 3 ,\\
c(w_ 2 \otimes w_ 1 )&= qx_5x_1^{-1} w_ 1 \otimes w_ 3 ,&
c(w_ 2 \otimes w_ 2 )&= q w_ 2 \otimes w_ 2 ,\\
c(w_ 2 \otimes w_ 3 )&= q^2x_5(x_1x_2)^{-1} w_ 3 \otimes w_ 4 ,&
c(w_ 2 \otimes w_ 4 )&= x_3x_5x_1^{-1} w_ 4 \otimes w_ 1 ,\\
c(w_ 3 \otimes w_ 1 )&= x_3x_5x_2^{-1} w_ 1 \otimes w_ 4 ,&
c(w_ 3 \otimes w_ 2 )&= x_3 w_ 2 \otimes w_ 1 ,\\
c(w_ 3 \otimes w_ 3 )&= q w_ 3 \otimes w_ 3 ,&
c(w_ 3 \otimes w_ 4 )&= x_1x_3q^{-1} w_ 4 \otimes w_ 2 ,\\
c(w_ 4 \otimes w_ 1 )&= x_1x_2x_3q^{-2} w_ 1 \otimes w_ 2 ,&
c(w_ 4 \otimes w_ 2 )&= x_2 w_ 2 \otimes w_ 3 ,\\
c(w_ 4 \otimes w_ 3 )&= qx_2x_5^{-1} w_ 3 \otimes w_ 1 ,&
c(w_ 4 \otimes w_ 4 )&= q w_ 4 \otimes w_ 4.
\end{align*}
If $q=-1$, $x_3x_5=1$, then $\dim\mathfrak{B}(W_6)=72$ and 
\begin{align*}
w_1^2=w_2^2=w_3^2=w_4^2=0,\quad
w_4w_1-x_1x_2x_3 w_1w_2-x_1 w_2w_4=0,\\
w_1w_3-x_1 w_3w_2+x_1x_3 w_2w_1=0,\quad
w_3w_4-x_1x_2x_3 w_2w_3+x_1x_3 w_4w_2=0,
\\
w_1w_4-x_5 w_4w_3-x_2 w_3w_1=0,\quad
(w_2w_1w_4)^2+(w_1w_4w_2)^2+(w_4w_2w_1)^2=0.
\end{align*}
\end{example}

\begin{example}
Let $(x_{1} x_{3} x_{7}^{3})^2=q^{10}$, $qx_1x_2x_3x_7\in\k^\times$. 
$W_7=\bigoplus_{i=1}^4\k w_i$ is a braided vector space,  with the braiding given by
\begin{align*}
c(w_ 1 \otimes w_ 1 )&= x_1 w_ 3 \otimes w_ 2 ,&
c(w_ 1 \otimes w_ 2 )&= q w_ 1 \otimes w_ 2 ,\\
c(w_ 1 \otimes w_ 3 )&= x_3 w_ 2 \otimes w_ 2 ,&
c(w_ 1 \otimes w_ 4 )&= x_1x_3x_7^2q^{-3} w_ 4 \otimes w_ 2 ,\\
c(w_ 2 \otimes w_ 1 )&= q w_ 2 \otimes w_ 1 ,&
c(w_ 2 \otimes w_ 2 )&= x_1x_7^2(x_2x_3)^{-1} w_ 4 \otimes w_ 1 ,\\
c(w_ 2 \otimes w_ 3 )&= x_7 w_ 3 \otimes w_ 1 ,&
c(w_ 2 \otimes w_ 4 )&= x_2x_3x_7(qx_1)^{-1} w_ 1 \otimes w_ 1 ,\\
c(w_ 3 \otimes w_ 1 )&= x_1^3x_7^5(q^5x_2^2)^{-1} w_ 4 \otimes w_ 4 ,&
c(w_ 3 \otimes w_ 2 )&= x_1x_7^2(qx_2)^{-1} w_ 2 \otimes w_ 4 ,\\
c(w_ 3 \otimes w_ 3 )&= qx_7x_2^{-1} w_ 1 \otimes w_ 4 ,&
c(w_ 3 \otimes w_ 4 )&= q w_ 3 \otimes w_ 4 ,\\
c(w_ 4 \otimes w_ 1 )&= x_2x_3x_7q^{-2} w_ 1 \otimes w_ 3 ,&
c(w_ 4 \otimes w_ 2 )&= x_2 w_ 3 \otimes w_ 3 ,\\
c(w_ 4 \otimes w_ 3 )&= q w_ 4 \otimes w_ 3 ,&
c(w_ 4 \otimes w_ 4 )&= q^3x_2^2x_3(x_1^2x_7^3)^{-1} w_ 2 \otimes w_ 3.
\end{align*}
If $q=-1=x_1x_3x_7^3$, then $\dim \mathfrak{B}(W_7)=72$ and 
\begin{align*}
w_4w_2-x_7 w_1w_4-x_2 w_3^2=0,\quad
w_1^2-x_1 w_3w_2+x_1(x_2x_3x_7)^{-1} w_2w_4=0,\\
w_1w_2=w_2w_1=w_3w_4=w_4w_3=0,\quad
w_1w_3-x_3 w_2^2-(x_2x_3x_7)^{-1} w_4w_1=0,\\
w_2w_3-x_7 w_3w_1+x_2^{-2}x_3^{-3}x_7^{-3}w_4^2=0,\\
(w_1^2w_3)^2+(w_1w_3w_1)^2 +x_1x_2^{-1}w_2^2w_3w_1^2w_4+x_1 w_2w_3w_1^2w_3^2+(w_3w_1^2)^2=0.
\end{align*}
\end{example}

\begin{example}
Let $qx_1x_3x_4x_7\in \k^\times$ and $q^4=(x_4x_7)^2$.
$W_8=\bigoplus_{i=1}^4\k w_i$ is a braided vector space,  with the braiding given by
 \begin{align*}
 c(w_ 1 \otimes w_ 1 )&= x_1 w_ 2 \otimes w_ 3 ,&
c(w_ 1 \otimes w_ 2 )&= q w_ 1 \otimes w_ 2 ,\\
c(w_ 1 \otimes w_ 3 )&= x_3 w_ 4 \otimes w_ 4 ,&
c(w_ 1 \otimes w_ 4 )&= x_4 w_ 3 \otimes w_ 1 ,\\
c(w_ 2 \otimes w_ 1 )&= q w_ 2 \otimes w_ 1 ,&
c(w_ 2 \otimes w_ 2 )&=x_3^3x_4^3(qx_1^2x_7^2)^{-1} w_ 1 \otimes w_ 4 ,\\
c(w_ 2 \otimes w_ 3 )&= x_7 w_ 4 \otimes w_ 2 ,&
c(w_ 2 \otimes w_ 4 )&=x_3^2x_4^4(q^3x_1x_7)^{-1} w_ 3 \otimes w_ 3 ,\\
c(w_ 3 \otimes w_ 1 )&=q^2x_1^2x_7^3(x_3^3x_4^3)^{-1} w_ 2 \otimes w_ 2 ,&
c(w_ 3 \otimes w_ 2 )&=x_3x_4^2(qx_1)^{-1} w_ 1 \otimes w_ 3 ,\\
c(w_ 3 \otimes w_ 3 )&=q^3x_7(x_3x_4^2)^{-1} w_ 4 \otimes w_ 1 ,&
c(w_ 3 \otimes w_ 4 )&= q w_ 3 \otimes w_ 4 ,\\
c(w_ 4 \otimes w_ 1 )&=qx_1x_7(x_3x_4)^{-1} w_ 2 \otimes w_ 4 ,&
c(w_ 4 \otimes w_ 2 )&=qx_4x_1^{-1} w_ 1 \otimes w_ 1 ,\\
c(w_ 4 \otimes w_ 3 )&= q w_ 4 \otimes w_ 3 ,&
c(w_ 4 \otimes w_ 4 )&=q^2x_1x_7(x_3^2x_4)^{-1} w_ 3 \otimes w_ 2.
 \end{align*}
 If $q=-1$, $x_4x_7=1$, then $\dim \mathfrak{B}(W_8)=72$ and 
 \begin{align*}
 w_1w_2=w_2w_1=w_3w_4=w_4w_3=0,\quad
 w_1^2-x_1w_2w_3+x_1x_7 w_4w_2=0,\\
 w_1w_3-x_3w_4^2+x_1x_7^2x_3^{-1}w_3w_2=0,\quad
 w_1w_4-x_7^{-1}w_3w_1+x_1^2x_7^5x_3^{-3}w_2^2=0,\\
 w_2w_4+x_3^2(x_1x_7^5)^{-1}w_3^2+x_3(x_1x_7^2)^{-1}w_4w_1=0,\\
 w_1w_3w_2w_2w_4w_1+w_2w_2w_4w_1w_1w_3+w_4w_1w_1w_3w_2w_2=0.
 \end{align*}
 \end{example}

\section*{Acknowledgements}
The author thanks the referees for careful reading and helpful comments on the writing of this paper. 
Their suggestions generalize the results of $\mathfrak{B}(W_{X,r})$ 
from  special cases of  permutation solutions to 
arbitrary permutation solution $(X,r)$. This inspires the author to generalize the results
to all non-degenerate involutive solutions.
This work was partially supported by
Foundation of Jiangxi Educational Committee (GJJ191681).

\bibliographystyle{CommAlg2022}
\bibliography{suzuki}
%\bibitem[GAP]{GAP4}
%  The GAP~Group, \emph{GAP -- Groups, Algorithms, and Programming, 
%  Version 4.11.1}; 2021, \url{https://www.gap-system.org}.
\end{document}